\documentclass{amsart}

\usepackage{amssymb,amsmath,amsthm,color}

\usepackage[bookmarksopen=true,bookmarksnumbered=true, bookmarksopenlevel=2, colorlinks=true,linkcolor=mblue,
citecolor=mblue,urlcolor=mblue, pdfstartview=FitH]{hyperref}

\usepackage{enumerate}

\usepackage[all]{xy}

\definecolor{mblue}{rgb}{0,0,.8}

\newcommand{\N}{\mathbb N}
\newcommand{\Z}{\mathbb Z}
\newcommand{\Q}{\mathbb Q}

\newcommand{\F}{\mathbb F}

\newcommand{\R}{\mathbb R}

\newcommand\U[0]{\mathcal{U}}
\newcommand\T[0]{\mathcal{T}}
\newcommand\HH[0]{\mathcal{H}}
\newcommand\sh[1]{\mathcal{#1}}
\newcommand\mor[3][]{#2 \xrightarrow{#1} #3}

\newcommand\es[0]{e^\ast}
\newcommand\Es[0]{E^\ast}

\newtheorem{thm}{Theorem}[section]
\newtheorem{lem}[thm]{Lemma}
\newtheorem{remark}[thm]{Remark}
\newtheorem{dfn}[thm]{Definition}
\newtheorem{prop}[thm]{Proposition}
\newtheorem{cor}[thm]{Corollary}

\newtheorem{condition}{Condition}

\newtheorem*{theorem}{Theorem}

\newtheorem{thmx}{Theorem}

\newcounter{tmp}

\newtheorem{thmy}{Proposition}

\DeclareMathOperator{\GL}{GL} \DeclareMathOperator{\SL}{SL}

\begin{document}

\title[Eisenstein series and overconvergence]{Eisenstein series, $p$-adic modular functions, and overconvergence}
\author[Ian Kiming, Nadim Rustom]{Ian Kiming, Nadim Rustom}
\address[Ian Kiming]{Department of Mathematical Sciences, University of Copenhagen, Universitetsparken 5, DK-2100 Copenhagen \O ,
Denmark.}
\email{kiming@math.ku.dk}
\address[Nadim Rustom]{Tryg Forsikring A/S, Klausdalsbrovej 601, 2750 Ballerup, Denmark}
\email{restom.nadim@gmail.com}



\begin{abstract} Let $p$ be a prime $\ge 5$. We establish explicit rates of overconvergence for some members of the ``Eisenstein family'', notably for the $p$-adic modular function $V(E_{(1,0)}^{\ast})/E_{(1,0)}^{\ast}$ ($V$ the $p$-adic Frobenius operator) that plays a pi\-votal role in Coleman's theory of $p$-adic families of modular forms. The proof goes via an in-depth analysis of rates of overconvergence of $p$-adic modular functions of form $V(E_k)/E_k$ where $E_k$ is the classical Eisenstein series of level $1$ and weight $k$ divisible by $p-1$. Under certain conditions, we extend the latter result to a vast generalization of a theorem of Coleman--Wan regarding the rate of overconvergence of $V(E_{p-1})/E_{p-1}$. We also comment on previous results in the literature. These include applications of our results for the primes $5$ and $7$.
\end{abstract}

\maketitle

\section{Introduction}\label{sec:intro} Everywhere in what follows, $p$ will denote a fixed prime $\ge 5$. Also, $K$ will denote a finite extension of $\Q_p$ with ring of integers $O$. By $v_p$ we denote the $p$-adic valuation of $\overline{\Q}_p$, normalized so that $v_p(p)=1$.
\smallskip

A major part of the motivation behind this paper is to be found in the following quote from Coleman and Mazur's paper \cite{coleman_mazur_eigencurve} constructing the eigencurve: ``For fuller control of the geometry of the eigencurve it would be useful to know explicit affinoid regions ... on which some, or all of the functions $E_l(q)$ converge." (Cf.\ \cite[p.\ 43]{coleman_mazur_eigencurve}.) Let us explain in detail:

The functions mentioned in the quote are functions in the so-called Eisenstein family and are $p$-adic modular functions constructed from $p$-adic Eisenstein series. The principal question is to establish explicit rates of overconvergence for them.

Let $s\in\N$. Then
$$
\Es_{(s,0)}(q) = 1 - \frac{2s}{B_{s,\tau^{-s}}} \sum_{n=1}^{\infty} \left( \sum_{\substack{d\mid n \\ p\nmid d}} d^{s-1} \tau(d)^{-s} \right) q^n
$$
where $\tau$ is the composition of reduction modulo $p$ and the Teichm\"uller character and $B_{s,\tau^{-s}}$ is a generalized Bernoulli number, is the $q$-expansion of a classical modular form of weight $s$ on $\Gamma_0(p)$ with nebentypus $\tau^{-s}$. These Eisenstein series are normalized versions of more general $p$-adic Eisenstein series $G_{(s,i)}^{\ast}$ where $s$ more generally is allowed to be a $p$-adic integer and $i\in\Z/\Z(p-1)$, as constructed by Serre, cf.\ \cite[1.6]{serre_zeta}.

From the $\Es_{(s,0)}$ we can form the $p$-adic modular functions
$$
\frac{V(\Es_{(s,0)})}{\Es_{(s,0)}}
$$
where $V$ denotes the $p$-adic Frobenius operator, acting on $q$-expansions as $q\mapsto q^p$. These functions are specializations to integral weights of the ``Eisenstein family'', cf.\ \cite{coleman_mazur_eigencurve}, \cite{coleman_eisenstein}.

The function $\mathcal{E} = V(\Es_{(1,0)})/\Es_{(1,0)}$ is particularly interesting and plays a pivotal role in Coleman's fundamental paper \cite{coleman_banach} that -- among other things -- establishes the existence of $p$-adic analytic families of eigenforms passing through classical eigenforms of finite $p$-slope. The reason for this is that the action of Atkin's $U$-operator on forms of some weight $k$ via ``Coleman's trick'' is conjugate to the action of a ``twisted '' $U$-operator, $U_{(k)}(\cdot) := U(\mathcal{E}^k \cdot)$ acting on overconvergent modular functions. For this to work, it is necessary to know that $\mathcal{E}$ is overconvergent with some positive rate of overconvergence, and Coleman proves that. He does not give an explicit positive number $\rho$ such that the rate of overconvergence is $\ge \rho$.

However, such explicit information about the rate of overconvergence of $\mathcal{E}$ or other closely related functions seems to be the best tool that we have for the more detailed study of the geometry of the Coleman--Mazur eigencurve (\cite{coleman_mazur_eigencurve}): cf.\ the papers \cite{emerton_2-adic}, \cite{smithline_thesis}, \cite{cst}, \cite{buzzard_kilford_2-adic}, \cite{roe_3-adic} each of which is concerned with either the prime $2$ or $3$. In \cite{coleman_eisenstein} Coleman formulates a general conjecture on ``analytic continuation'' of the Eisenstein family and proves that conjecture for the primes $2$ and $3$ as a consequence of central statements of the papers \cite{buzzard_kilford_2-adic} and \cite{roe_3-adic} respectively. These central results are precisely statements about rates of overconvergence for modular functions as in Theorem \ref{thm:es} where however the fixed prime is $2$ or $3$. See section \ref{sec:literature} below for a summary of these results as reformulated in our language. Section \ref{sec:literature} will also sketch some sample applications of our results for the primes $5$ and $7$.

We are not aware of any explicit information in the literature about the rate of overconvergence of the function $\mathcal{E}$ if $p$ is a prime $\ge 5$. The principal aim of this paper is to provide such information, valid for any prime $p\ge 5$.

In order to formulate our results, let us introduce some notation. Our (tame) level will (almost) always be $1$ and so reference to the level will be dropped from (almost) all of our notation. If $k$ is a non-negative integer and $r\in O$, denote by $M_k(\cdot,r)$ with $\cdot = O$ or $K$ the $O$-module or $K$-vector space of $r$-overconvergent modular forms of weight $k$ and tame level $1$, holomorphic at $\infty$. Thus, $M_k(K,1)$ can be identified with Serre $p$-adic modular forms of weight $k$ (\cite{serre_zeta}). To formulate our results we have found it convenient to introduce the $O$-module
$$
M_k(O,\ge \rho)
$$
for $\rho \in \Q \cap [0,1]$ as the $O$-module consisting of forms $f$ such that $f\in M_k(O,r)$ for some $r$ and such that if
$$
f = \sum_{i=0}^{\infty} \frac{b_i}{E_{p-1}^i}
$$
is a ``Katz expansion'' of $f$ then we have $v_p(b_i) \ge \rho i$ for all $i$. Alternatively, we have $f\in M_k(O,\ge \rho)$ if, whenever $K'/K$ is a finite extension with ring of integers $O'$ and $r'\in O'$ is such that $0\le v_p(r') < \rho$, then $f\in M_k(O',r')$. For the formal introduction of ``Katz expansions'' and this definition, see the beginning of section \ref{sec:topology_overconv_mod_fcts} below.

Our first theorem is now the following.

\begin{thmx}\label{thm:es} Let $s\in\N$. Then:
$$
\frac{V(E_{(s,0)}^{\ast})}{E_{(s,0)}^{\ast}} \in \frac{1}{p} M_0\left( \Z_p,\ge \frac{1}{p+1} \right) ,
$$
and also
$$
\frac{V(E_{(s,0)}^{\ast})}{E_{(s,0)}^{\ast}} \in M_0\left( \Z_p,\ge \frac{2}{3} \cdot \frac{1}{p+1} \right) .
$$
\end{thmx}

The way we approach this theorem is via a study of classical Eisenstein series. As usual, for an even integer $k\ge 4$ denote by $E_k$ the classical, normalized Eisenstein series of weight $k$ and level $1$:
$$
E_k(q) = 1 - \frac{2k}{B_k} \sum_{n=1}^{\infty} \sigma_{k-1}(n) q^n
$$
with, as usual, $\sigma_{k-1}$ the power-divisor function $\sigma_{k-1}(n) := \sum_{d\mid n} d^{k-1}$.

We shall only be concerned with these Eisenstein series for even integers $k\ge 4$ that are divisible by $p-1$. If one begins to ask questions about rate of overconvergence of a $p$-adic modular function of form $V(E_k)/E_k$, the only result in the literature for primes $p\ge 5$ seems to be the following theorem. We shall refer to the theorem as the ``Coleman--Wan'' theorem as it appears in Wan's paper \cite{wan}, but is there attributed to Coleman.

\begin{theorem} (Coleman--Wan, cf.\ \cite[Lemma 2.1]{wan}) The modular function
$$
\frac{V(E_{p-1})}{E_{p-1}}
$$
is a $1$-unit in $M_0(O,r)$ whenever $r\in O$ with $v_p(r) < \frac{1}{p+1}$.
\end{theorem}

Here, a ``$1$-unit'' of $M_0(O,r)$ is an element of form $1+a\cdot f$ where $f\in M_0(O,r)$ and $a\in O$ with $v_p(a)>0$. Such an element is obviously invertible in the ring $M_0(O,r)$, and the inverse is again a $1$-unit. This notion will also be employed in slightly different situations in what follows, but the definition will be clear in each case.

We prove for general primes $p\ge 5$ the following theorem that may be of independent interest.

\begin{thmx}\label{thm:V(E)/E_oc} Let $k\ge 4$ be divisible by $p-1$. Then:
$$
\frac{V(E_k)}{E_k} \in \frac{1}{p} M_0\left( \Z_p,\ge \frac{1}{p+1} \right) ,
$$
and also
$$
\frac{V(E_k)}{E_k} \in M_0\left( \Z_p,\ge \frac{2}{3} \cdot \frac{1}{p+1} \right) .
$$
\end{thmx}

See Remark \ref{rem:V(E_k)/E_k_1_unit} in section \ref{sec:proofs} for a reformulation of the second part of the Theorem in terms of $1$-units.

In addition to the Eisenstein series $E_k$ we shall need their $p$-deprived counterparts:
$$
\Es_k := \frac{E_k - p^{k-1}V(E_k)}{1-p^{k-1}}
$$
($k$ even integer $\ge 4$.) Thus, $\Es_k$ is an Eisenstein series for $\Gamma_0(p)$, but we also have $\Es_k = \Es_{(k,0)}$ if $k\equiv 0\pmod{p-1}$. Thus, Theorem \ref{thm:es} includes a statement about the $p$-adic modular function $V(\Es_k)/\Es_k$ for $k\equiv 0\pmod{p-1}$.

\begin{dfn}\label{def:en_esn} For any $n\in\N$ define the $p$-adic modular functions
$$
e_n := \frac{E_{n(p-1)}}{E_{p-1}^n} \quad \mbox{and} \quad \es_n := \frac{\Es_{n(p-1)}}{E_{p-1}^n} .
$$
\end{dfn}

\begin{thmx}\label{thm:es_n_oc} Let $n\in\N$. We have
$$
e_n, \es_n \in \frac{1}{p} M_0\left( \Z_p,\ge \frac{p}{p+1} \right) ,
$$
and also
$$
e_n, \es_n \in M_0\left( \Z_p,\ge \frac{2}{3} \cdot \frac{p}{p+1} \right) .
$$
\end{thmx}

Let us inform the reader at this point that the first statements of each of Theorems \ref{thm:es}, \ref{thm:V(E)/E_oc}, \ref{thm:es_n_oc} are the most precise parts of the theorems. The second statements follow ultimately from the first statement of Theorem \ref{thm:es_n_oc} in combination with the congruence $e_n \equiv 1\pmod{p^2}$ that we will using in arguments in section \ref{sec:classical_eisenstein} below.

Numerical examples in section \ref{subsec:numerical_examples} below will show that the factor $\frac{1}{p}$ in the first statements of Theorems \ref{thm:V(E)/E_oc}, \ref{thm:es_n_oc} can not in general be removed.

Noting this difference between Theorem \ref{thm:V(E)/E_oc} and the Coleman--Wan theorem, i.e., the presence of the fraction $\frac{1}{p}$ in Theorem \ref{thm:V(E)/E_oc}, this raises the question as to why the weight $k=p-1$ in the Coleman--Wan theorem is special. We may ask whether there is a class of weights for which we have a direct generalization of the Coleman--Wan theorem. For certain primes $p$ we have an answer to this question, and as a consequence also an improvement of Theorem \ref{thm:es} for certain $s$.

In order to formulate these final of our results, we will define a function that has been called the ``$p$-adic weight'' $\delta_p(n)$ of a number $n\in\N$: if $n$ has $p$-adic expansion
$$
n = \sum_{i\geq 0} a_i p^i
$$
with $a_i \in \{0, 1, \ldots, p-1\}$, we define
$$
\delta_p(n) := \sum_{i\geq 0} a_i .
$$

Consider now the following condition on $p$.

\begin{condition}\label{condition:e_n_for_small_n} We have
$$
e_n \in M_0\left(\Z_p , \geq \frac{p}{p+1}\right)
$$
for all $n \in \{1, 2, \ldots, p\}$.
\end{condition}

In section \ref{sec:comp} below we explain how Condition \ref{condition:e_n_for_small_n} can be checked computationally for a given prime $p$. We have verified:

\setcounter{tmp}{\value{thmx}}

\setcounter{thmy}{\thetmp}

\begin{thmy}\label{prop:comp} Condition \ref{condition:e_n_for_small_n} holds for $5\le p\le 97$.
\end{thmy}

For primes $p$ such that Condition \ref{condition:e_n_for_small_n} is true we have the following vast generalization of the Coleman--Wan theorem above as well as an improvement of Theorem \ref{thm:es} for certain $s$, including $s=1$:

\stepcounter{thmx}

\begin{thmx}\label{thm:e_n_oc_conditional} Assume that Condition \ref{condition:e_n_for_small_n} is true for $p$. Then we have
$$
e_n \in M_0\left(\Z_p , \geq \frac{p}{p+1}\right)
$$
for all $n\in\N$ such that $\delta_p(n(p-1)) = p-1$.

As consequences we have the following. For any integer $k\ge 4$ with $k\equiv 0 \pmod{p-1}$ and $\delta_p(k) = p-1$ we have
$$
\frac{V(E_k)}{E_k} \in M_0\left( \Z_p,\ge \frac{1}{p+1} \right) .
$$

If $s\in\N$ with $\delta_p(s) < p-1$ then
$$
\frac{V(\Es_{(s,0)})}{\Es_{(s,0)}} \in M_0\left( \Z_p,\ge \frac{1}{p+1} \right) .
$$
\end{thmx}

It would be tempting to conjecture that Condition \ref{condition:e_n_for_small_n} holds for all primes $p\ge 5$, but we do not have any real reason for doing so beyond the computational work in section \ref{sec:comp}.

Finally, in section \ref{sec:literature} we summarize some previous results for the primes $2$ and $3$ and describe some applications of our results, specifically of the statement about $V(\Es_{(1,0)})/\Es_{(1,0)}$ of Theorem \ref{thm:e_n_oc_conditional} for the primes $p=5,7$.

\section{Preliminaries on overconvergent modular functions}\label{sec:topology_overconv_mod_fcts} Let us first state formal definitions of the $O$-modules (or, $K$-vector spaces) $M_k(\cdot,r)$ and $M_k(\cdot,\ge \rho)$, where $\cdot = O$ or $K$, that we will be working with. We shall confine ourselves to weights $k$ that are divisible by $p-1$, and in fact later on will assume $k=0$, but of course the basic facts are analogous for other weights. Although all of our results are concerned with forms of (tame) level $1$, for the proofs we will occasionally have to connect with the modular definition of overconvergent modular forms (modular forms with growth conditions in the language of \cite{katz_padic}), and will be referring to Katz' original paper \cite{katz_padic} as well as Gouv\^{e}a's book \cite{gouvea_p-adic_modular_forms}. For this reason we need to introduce these concepts in a more general setting.

Let $N\in\N$ be a natural number with $p\nmid N$. Let $k$ be a non-negative integer with $k\equiv 0 \pmod{p-1}$. We will denote by $M_k(N,\cdot)$ where $\cdot = O$ or $K$ the $O$-module or $K$-vector space consisting of classical modular forms of weight $k$, level $\Gamma(N)$, and coefficients in $O$ resp.\ $K$. If $N=1$ we will suppress $N$ from the notation and just write $M_k(\cdot)$. A few times we will have to consider modular forms on other congruence subgroups $\Gamma$ and will then write $M_k(\Gamma,O)$ for modular forms of weight $k$ on $\Gamma$ and coefficients in $O$.

Suppose now that $N\ge 3$ and that $i > 0$ is an integer. Then we have a (non-canonical) splitting
$$
M_{i(p-1)}(N,\Z_p) = E_{p-1} \cdot M_{(i-1)(p-1)}(N,\Z_p) \oplus B^{(0)}_i(N,\Z_p)
$$
with a free $\Z_p$-module $B^{(0)}_i(N,\Z_p)$. This corresponds to the geometric statement of \cite[Lemma 2.6.1]{katz_padic}. We make a fixed choice of these $B^{(0)}_i(N,\Z_p)$ once and for all and also define $B_0^{(0)}(N,\Z_p) := \Z_p$. By tensoring with $O$ or $K$ we obtain $B^{(0)}_i(N,O)$ and $B^{(0)}_i(N,K)$ and keep the above splitting, now with coefficients in $O$ and $K$, respectively.

If $k$ is now a non-negative integer $\equiv 0\pmod{p-1}$, we define $B^{(k)}_0(N,O) := M_k(N,O)$ and $B^{(k)}_i(N,O) := B^{(0)}_{i+k/(p-1)}(N,O)$ for $i>0$, and the same for coefficients in $K$. We then have
$$
B^{(k+t(p-1))}_{i-t}(N,\cdot) = B^{(k)}_i(N,\cdot)
$$
for $t<i$, as well as the splitting
$$
M_{k+i(p-1)}(N,\cdot) = E_{p-1} \cdot M_{k+(i-1)(p-1)}(N,\cdot) \oplus B^{(k)}_i(N,\cdot)
$$
for $i>0$ ($\cdot = O$ or $K$.)

If now $r\in O$, the content of \cite[Proposition 2.6.2]{katz_padic} is that elements $f$ of the $O$-module $M_k(N,O,r)$ of $r$-overconvergent modular forms of weight $k$, level structure $\Gamma(N)$, and coefficients in $O$ can be identified with sums of the form
$$
f = \sum_{i=0}^{\infty} \frac{b_i}{E_{p-1}^i} \leqno{(\ast)}
$$
where $b_i \in B^{(k)}_i(N,O)$ for all $i$ and satisfy
$$
v_p(b_i) \ge i v_p(r) \quad \mbox{for all $i$}
$$
as well as
$$
v_p(b_i) - i v_p(r) \rightarrow \infty \quad \mbox{for $i\rightarrow \infty$} .
$$

(The $b_i$ are classical modular forms. The condition $v_p(b_i) \ge \rho$ means that all coefficients of the $q$-expansions of $b_i$ have valuation $\ge \rho$.)

Tensoring with $K$ gives us the space $M_k(N,K,r)$ that can be identified with sums $(\ast)$ where again $b_i \in B^{(k)}_i(N,O)$ for all $i$, but are only required to satisfy $v_p(b_i) - i v_p(r) \rightarrow \infty$ for $i\rightarrow \infty$.

Keeping all of the above notation, but choosing specifically $N=3$, the group $\GL_2(\Z/N\Z) \cong \SL_2(\Z)/\Gamma(N)$ has order prime to $p$ (as $p\ge 5$.) This group acts on the splitting of $M_{k+i(p-1)}(N,\cdot)$ and we have a projector onto invariants. Defining
$$
B^{(k)}_i(\cdot) := B^{(k)}_i(3,\cdot)^{\GL_2(\Z/3\Z)}
$$
with $\cdot = O$ or $K$, the content of \cite[Proposition 2.8.1]{katz_padic} is that the $O$-module $M_k(O,r)$ of $r$-overconvergent modular forms $f$ of weight $k$, level $1$, and coefficients in $O$ can be identified with sums as in $(\ast)$ but where now $b_i \in B^{(k)}_i(O)$ for all $i$ and satisfy $v_p(b_i) \ge i v_p(r)$ for all $i$ as well as $v_p(b_i) - i v_p(r) \rightarrow \infty$ for $i\rightarrow \infty$. We get the vector space $M_k(K,r)$ by tensoring with $K$.

As we have fixed our choices of the $B^{(k)}_i$, for an overconvergent form $f$ we shall refer to the expression $(\ast)$ as ``the'' Katz expansion of $f$. It is uniquely determined by $f$ by virtue of Propositions 2.6.2 and 2.8.1 of \cite{katz_padic}.

As the weight $k$ in the above will be $k=0$ almost all of the time in what follows, we will suppress the reference to the weight when it is $0$ and write
$$
B_i(\cdot) := B^{(0)}_i(\cdot) .
$$

Suppose that $k\in\N$ is divisible by $p-1$, say $k=n(p-1)$, and that $f\in M_k(O)$. From the above we see that
$$
f = \sum_{i=0}^{n} E_{p-1}^{n-i} \cdot b_i
$$
with $b_i\in B_i(O)$ for $i=0,\ldots,n$. The modular function $f/E_{p-1}^n$ is $r$-overconvergent for some $r$, certainly for $r=1$. We see what the shape of its Katz expansion is:

\begin{prop}\label{prop:katz_exp_classical_form} Suppose that $k\in\N$ is divisible by $p-1$, let $n:=k/(p-1)$, and let $f\in M_k(O)$. Then the Katz expansion of the modular function $f/E_{p-1}^n$ has form
$$
\frac{f}{E_{p-1}^n} = \sum_{i=0}^{n} \frac{b_i}{E_{p-1}^i}
$$
with $b_i\in B_i(O)$ for $i=0,\ldots,n$.
\end{prop}

\begin{dfn}\label{dfn:overconv_greater_than_or_equal_rho} Let $k\in\N$ be divisible by $p-1$. For $\rho \in \Q \cap [0,1]$ define the $O$-module $M_k(O,\ge \rho)$ as the $O$-module of forms $f$ such that $f\in M_k(O,r)$ for some $r$, and such that for the Katz expansion $f = \sum_{i=0}^{\infty} \frac{b_i}{E_{p-1}^i}$ with $b_i \in B^{(k)}_i(O)$ we have
$$
v_p(b_i) \ge \rho i
$$
for all $i$.
\end{dfn}

The following proposition gives a simple, alternative definition of $M_k(O,\ge \rho)$.

\begin{prop}\label{prop:overconv_greater_than_or_equal_rho} Let $k\in\N$ be divisible by $p-1$. Suppose that $\rho \in \Q \cap [0,1]$ and that $f\in M_k(O,r)$ for some $r$. Then the following are equivalent.
\begin{enumerate}[(i)]
\item $f\in M_k(O,\ge \rho)$.
\item Whenever $K'/K$ is a finite extension with ring of integers $O'$ and $r'\in O'$ is such that $0\le v_p(r') < \rho$ then $f\in M_k(O',r')$.
\item There is a sequence of finite extensions $K'/K$ with rings of integers $O'$ as well as elements $r'\in O'$ such that $v_p(r')$ converges from below towards $\rho$ and such that $f\in M_k(O',r')$ for each of these $O',r'$.
\end{enumerate}
\end{prop}

\begin{proof} Consider the Katz expansion
$$
f = \sum_{i=0}^{\infty} \frac{b_i}{E_{p-1}^i}
$$
with $b_i\in B^{(k)}_i(O)$ for all $i$.

Suppose that $f\in M_k(O,\ge \rho)$, let $K'/K$ be a finite extension with ring of integers $O'$, and let $r'\in O'$ be such that $0\le v_p(r') < \rho$. We then have $v_p(b_i) \ge \rho i > i v_p(r')$ for all $i$ and furthermore,
$$
v_p(b_i) - i v_p(r') = v_p(b_i) - i\rho + i (\rho - v_p(r')) \ge i (\rho - v_p(r')) \rightarrow \infty
$$
for $i\rightarrow \infty$ as $v_p(r') < \rho$. Thus, $f\in M_k(O',r')$.

On the other hand, given $\rho$ we can find a sequence of finite extensions $K'/K$ with rings of integers $O'$ as well as elements $r'\in O'$ such that $v_p(r')$ converges from below towards $\rho$. If we then have $f\in M_k(O',r')$ for these $r'$, we find that
$$
v_p(b_i) \ge i v_p(r')
$$
for all $i$ and each $r'$, and we deduce $v_p(b_i) \ge \rho i$ for all $i$ whence $f\in M_k(O,\ge \rho)$.
\end{proof}

\begin{remark}\label{rem:higher_level_=>rho} Of course we could have made Definition \ref{dfn:overconv_greater_than_or_equal_rho} more generally at a level $N$ prime to $p$. If one does that, one has the analogue of Proposition \ref{prop:overconv_greater_than_or_equal_rho}, proved by the same proof.
\end{remark}

In what follows below, we need the following observation. It follows instantly from the work of Katz \cite{katz_padic}, but we shall give the argument for the convenience of the reader.

\begin{prop}\label{prop:congr_in_q-exp_implies_congr_for Katz_exp} Suppose that $f\in M_0(O,\ge \rho)$ with Katz expansion $\sum_{i\ge 0} \frac{b_i}{E_{p-1}^i}$. Let $m\in\N$.

If $f\equiv 0 \pmod{p^m}$ in the $q$-expansion then $v_p(b_i) \ge m$ for all $i$.
\end{prop}

\begin{proof} Let us consider $f$ as an element of $M_0(O,1)$. The hypothesis that $f\equiv 0 \pmod{p^m}$ in $q$-expansion now implies that $f\in p^m M_0(O,1)$ by the $q$-expansion principle \cite[Proposition 2.8.3]{katz_padic} (recall that we have $p\ge 5$.) Write $f=p^m g$ where $g\in M_0(O,1)$ has Katz expansion $\sum_i \frac{c_i}{E_{p-1}^i}$.

By \cite[Proposition 2.8.1]{katz_padic} there is an isomorphism:
$$
B^{{\rm rigid}}(O,1,0) \cong M_0(O,1)
$$
where $B^{{\rm rigid}}(O,1,0)$ is the $O$-module consisting of all formal sums $\sum_i a_i$ with $a_i \in B_i(O)$ such that $a_i \rightarrow 0$ for $i\rightarrow \infty$, given by
$$
\sum_i a_i \mapsto \sum_i \frac{a_i}{E_{p-1}^i} .
$$

We see that the formal sums $\sum_i b_i$ and $\sum_i p^m c_i$ have the same image under the isomorphism and hence must be equal. Consequently, $b_i = p^m c_i$ for all $i$, and the claim follows.
\end{proof}

Of course, in the proof we only use the injectivity part of the above isomorphism statement. One notices that in \cite[Proposition 2.8.1]{katz_padic} the only restriction is $p\ge 5$, but this is due to the level being $1$. For a more general situation one would have to refer to \cite[Proposition 2.6.2]{katz_padic} that has the more involved hypotheses of \cite[Theorem 2.5.1]{katz_padic}.

\begin{prop}\label{prop:conv_in_basis_closedness} Let $(f_n)_{n \in \Z_{\geq 0}}$ be a sequence in $M_0(O, 1)$ that converges in the $q$-expansions. Let $f\in M_0(O,1)$ denote the limit, and write
$$
f_n = \sum_{i\geq0} \frac{b_i^{(n)}}{E_{p-1}^i} ,\quad f = \sum_{i \geq 0} \frac{b_i}{E_{p-1}^i}
$$
for the Katz expansions of $f_n$ and $f$, respectively. Then we have the following.
\begin{enumerate}[(i)]
\item For each $i$ we have $\lim_{n \rightarrow \infty} b_i^{(n)} = b_i$ in $B_i(O)$.
\item If $f_n\in M_0(O, \geq \rho)$ for each $n$ then in fact $f\in M_0(O, \geq \rho)$. That is, the set $M_0(O,\geq \rho)$ is closed in $M_0(O, 1)$.
\end{enumerate}
\end{prop}

\begin{proof} (i) Let $m \in \N$. Then for all sufficient large $n$ we have
$$
f - f_n \in p^m M_0(O, 1).
$$

It then follows from Proposition \ref{prop:congr_in_q-exp_implies_congr_for Katz_exp} that we have
$$
b_i - b_i^{(n)} \in p^m B_i(O)
$$
for all $i$.

We see that
$$
b_i = \lim_{n \rightarrow \infty} b_i^{(n)}.
$$

\noindent (ii) By (i) we have $b_i = \lim_{n \rightarrow \infty} b_i^{(n)}$ in $B_i(O)$ for each $i \in \Z_{\geq 0}$. As now $f_n\in M_0(O, \geq \rho)$ for each $n$, by Definition \ref{dfn:overconv_greater_than_or_equal_rho} we have
$$
v_p(b_i^{(n)}) \ge \rho i
$$
for every $i, n \in \Z_{\geq 0}$. It follows that
$$
v_p(b_i) \ge \rho i
$$
for every $i \in \Z_{\geq 0}$. Thus $f \in M_0(O, \geq \rho)$, again by Definition \ref{dfn:overconv_greater_than_or_equal_rho}.
\end{proof}

\begin{cor}\label{cor:conv_in_qexp} Let $(f_n)_{n\in\Z_{\geq 0}}$ be a sequence in $M_0(O, \geq \rho)$ that converges in $q$-expansions. Then $(f_n)_{n\in\Z_{\geq 0}}$ also converges in $M_0(O, \geq \rho)$.
\end{cor}

\begin{proof} Immediate from Proposition \ref{prop:conv_in_basis_closedness} by embedding the $f_n$ into $M_0(O,1)$.
\end{proof}

\section{Classical Eisenstein series and overconvergence}\label{sec:classical_eisenstein}

\subsection{Twisted Hecke operators}\label{subsec:twisted_hecke_operators} Acting on Serre $p$-adic modular forms of weight $k$ we have the Atkin $U$ operator as well as the Hecke operator $T_{\ell}$ for any prime $\ell\neq p$:
$$
U, T_\ell: M_k(K, 1) \rightarrow M_k(K, 1)
$$
(cf.\ \cite[\S 2]{serre_zeta}.)

We will define and study twisted versions of these operators in weight $0$.

\begin{dfn}\label{def:twisted_operators} For $n \in \N$ and primes $\ell \neq p$, define
$$
\U_n(f) := \frac{U(fE_{p-1}^n)}{E_{p-1}^n}, \quad \T_{\ell, n}(f) := \frac{T_\ell(fE_{p-1}^n)}{E_{p-1}^n}
$$
for Serre $p$-adic modular functions $f\in M_0(K,1)$.
\end{dfn}

Notice that as the operators $U$ and $T_{\ell}$ commute we see that the operators $\U_n$ and $\T_{\ell,n}$ commute as well.

The reason for introducing these twisted operators is the fact that the modular function $\es_n$ (defined in Definition \ref{def:en_esn}) is an eigenfunction for all of them:

\begin{prop}\label{prop:es_n_eigenfct} Let $n \in \N$. We have $\T_{\ell,n} \es_n = (1+\ell^{n(p-1)-1}) \cdot \es_n$ and $\U_n \es_n = \es_n$.
\end{prop}

\begin{proof} This follows instantly from the definition of the twisted operators and the fact that $E_{n(p-1)}^{\ast}$ is an eigenform for the corresponding untwisted operators with the same eigenvalues. Thus,
$$
\mathcal{T}_{\ell,n} (e_n^{\ast}) = \frac{T_{\ell}(e_n^{\ast} E_{p-1}^n)}{E_{p-1}^n} = \frac{T_{\ell}(E_{n(p-1)}^{\ast})}{E_{p-1}^n} = c \cdot \frac{E_{n(p-1)}^{\ast}}{E_{p-1}^n} = c \cdot e_n^{\ast} ,
$$
with $c = (1+\ell^{n(p-1)-1})$, and we have a similar computation of $\U_n \es_n$ (with $c=1$.)
\end{proof}

It is now necessary for us to study the integrality properties of these operators when restricting to overconvergent modular functions. Of course, such questions for the ``untwisted operators'', i.e., the case $n=0$, were first studied in Katz' foundational paper \cite{katz_padic}, cf.\ \cite[Integrality Lemma 3.11.4]{katz_padic}.

For parts of the proofs of the following lemmas we will need to work with the original definition of overconvergent forms as functions on ``test objects'', and we will now first briefly recall the relevant definitions.

For $O$ as above, consider a $p$-adically complete and separated $O$-algebra $R$. Given $N\in\N$ prime to $p$ and $r\in O$, by an $r$-test object we mean a quadruple $(\sh{E}_{/R},\omega,\alpha_N,Y)$ consisting of an elliptic curve $\sh{E}$ over $R$, a nonvanishing differential $\omega$ on $\sh{E}$, a $\Gamma(N)$-level structure $\alpha_N$, and an element $Y\in R$ satisfying
$$
Y \cdot E_{p-1}(\sh{E},\omega) = r
$$
with $E_{p-1}$ the classical Eisenstein series of weight $p-1$, a lift of the Hasse invariant. If $N=1$ as will be the case in most of what follows, any level $N$ structure is trivial, and we will drop reference to it in our notation. Also, there is no reference to the level $N$ in ``$r$-test object'' as it will always be clear from the context what $N$ is.

An $r$-overconvergent modular form $f$ of weight $k$ is then a function of such quadruples (or, triples), assigning a value $f(\sh{E}_{/R},\omega,\alpha_N,Y) \in R$ to such a test object such that the value only depends on the isomorphism class of $(\sh{E}_{/R},\omega,\alpha_N,Y)$, the formation of $f(\sh{E}_{/R},\omega,\alpha_N,Y)$ commutes with base change, and we have
$$
f(\sh{E}_{/R},\lambda \omega,\alpha_N,\lambda^{1-p} Y)) = \lambda^{-k} f(\sh{E}_{/R},\omega,\alpha_N,Y)
$$
for $\lambda \in R^{\times}$.

\begin{lem}[Integrality Lemma]\label{integrality_lemma}  The operators $\U_n$ and $\T_{\ell, n}$ for $\ell \not = p$ restrict to operators
$$
\U_n: M_0\left(O, \geq \frac{1}{p+1} \right) \rightarrow \frac{1}{p} M_0\left(O, \geq \frac{p}{p+1}\right),\leqno{(i)}
$$
$$
\T_{\ell,n}: M_0\left(O, \geq \frac{p}{p+1}\right) \rightarrow M_0\left(O, \geq \frac{p}{p+1}\right).\leqno{(ii)}
$$
\end{lem}

\begin{proof}
\noindent (i) Referring back to Proposition \ref{prop:overconv_greater_than_or_equal_rho} we see that it suffices to show the following. Consider a finite extension $K'/K$ with ring of integers $O'$ and $r\in O'$ with $v_p(r) < \frac{1}{p+1}$. If then $f\in M_0(O',r)$ we have $p\U_n(f) \in M_0(O',r^p)$.

So, let now $O'$ and $r\in O'$ be as above, and let $f\in M_0(O',r)$. We now utilize ``Coleman's trick'', \cite[(3.3)]{coleman_classical}, which consists of the identity $\frac{1}{G}U(FG) = U(F\cdot \frac{G}{V(G)})$ that we will use in the special case $F=f$, $G=E_{p-1}^n$, giving
$$
\U_n(f) = \frac{1}{E_{p-1}^n} U(fE_{p-1}^n) = U \left( f\cdot \frac{E_{p-1}^n}{V(E_{p-1}^n)} \right) = U\left( f\cdot \left( \frac{E_{p-1}}{V(E_{p-1})} \right)^n \right)
$$
where we have $f\cdot (E_{p-1}/V(E_{p-1}))^n \in M_0(O',r)$ since $E_{p-1}/V(E_{p-1})$ is a $1$-unit in $M_0(O',r)$ by the Coleman--Wan theorem. The claim now follows by referring back to the ``Integrality Lemma'' \cite[3.11.4]{katz_padic} which will now show that $p\U_n(f) \in M_0(O',r^p)$. Technically, in order to refer to \cite[3.11.4]{katz_padic} we have to embed our situation into an auxiliary higher level $N$ prime to $p$, say $N=3$. However, this is unproblematic for the overall argument, cf.\ Remark \ref{rem:higher_level_=>rho}.
\smallskip

\noindent (ii) Consider a finite extension $K'/K$ with ring of integers $O'$ and $r\in O'$ with $v_p(r) < \frac{p}{p+1}$. Let $f\in M_0(O',r)$. We must show that $\T_{\ell,n}(f) \in M_0(O',r)$.

Let $(\sh{E}_{/R}, \omega, Y)$ be an $r$-test object with $R$ a $p$-adically complete and separated $O'$-algebra. Let $C_1, \ldots, C_{\ell+1}$ be the subgroups of $\sh{E}$ of order $\ell$, and
$$
\pi_i: \sh{E} \rightarrow \sh{E}/C_i
$$
be the quotient maps. Since $\ell \not = p$, the maps $\pi_i$ are \'etale and therefore $\sh{E}$ and $\sh{E}/C_i$ have the same Hasse invariant. Thus $E_{p-1}(\sh{E}, \omega)$ is a unit times $E_{p-1}(\sh{E}/C_i, \breve{\pi}_i^\ast\omega)$ and we find that
$$
``\frac{T_\ell (fE_{p-1}^n)}{E_{p-1}^n}"(\sh{E}, \omega, Y) = \ell^{n(p-1)-1}\sum_{i=1}^{\ell+1} \left( \frac{E_{p-1}(\sh{E}/C_i, \breve{\pi}_i^\ast \omega)}{E_{p-1}(\sh{E}, \omega)} \right)^n f(\sh{E}/C_i, \breve{\pi}_i^\ast \omega, \breve{\pi}_i^\ast Y)
$$
is an element of $R$, and we conclude $\T_{\ell,n}(f) \in M_0(O',r)$.
\end{proof}

If $N\in\N$ is prime to $p$ and $f\in M_k(\Gamma(N) \cap \Gamma_0(p),O)$ is a classical modular form, then under certain numerical restrictions on $N,p$ the theory of the canonical subgroup provides us with a natural embedding of $f$ into $M_k(N,O,r)$ when $v_p(r)<p/(p+1)$: the image $\tilde{f}$ of $f$ in $M_k(N,O,r)$ is given explicitly on an $r$-test object $(\sh{E}_{/R},\omega,\alpha_N,Y)$ by
$$
\tilde{f}(\sh{E}_{/R},\omega,\alpha_N,Y) = f(\sh{E}_{/R},\omega,\alpha_N,C)
$$
where $C$ is the canonical subgroup. As to the numerical restrictions, \cite[Theorem 3.2]{katz_padic} says that $N\ge 3$ is sufficient, but looking at the statements of \cite[Theorem 3.1]{katz_padic} one sees that this works in our cases for any $N\in\N$ as we always have $p\ge 5$. We shall henceforth simply write $f$ for the image of $f$ in $M_k(N,O,r)$ (for $v_p(r)<p/(p+1)$.)

If we have $N\ge 3$ prime to $p$ and either $k\ge 2$ or $k=0$, \cite[Corollary 2.6.3]{katz_padic} says that we have an injection $M_k(N,O,r) \hookrightarrow M_k(N,O,1)$ given by composition with the transformation of functors given by $(\sh{E}_{/R},\omega,\alpha_N,Y) \mapsto (\sh{E}_{/R},\omega,\alpha_N,rY)$.

\begin{lem}\label{lem:r^kV(f)} Let $f\in M_k(O)$ be a classical form of level $1$ and with $k\ge 2$. If $N\ge 3$ is prime to $p$, and if $r\in O$ with $v_p(r) < p/(p+1)$, we have
$$
r^k V(f) \in M_k(N,O,r) .
$$
\end{lem}

\begin{proof} Certainly, $V(f)$ is a Serre $p$-adic modular form of weight $k$, i.e., an element of $M_k(O,1) \subseteq M_k(N,O,1)$. The claim is that $r^k V(f)$ coincides with an element of $M_k(N,O,r)$ if this module is seen as a submodule of $M_k(N,O,1)$ via the embedding described before the statement of the lemma.

Recall again that since $v_p(r) < \frac{p}{p+1}$, if $(\sh{E}_{/R},\omega,\alpha_N,Y)$ is an $r$-test object then under the identification of $f$ with its image in $M_k(N,O,r)$ (after first embedding into $M_k(N,O)$) we have
$$
f(\sh{E},\omega,\alpha_N,Y) = f(\sh{E},\omega,\alpha_N,C)
$$
with $C$ the canonical subgroup of $\sh{E}$.

We now refer to (the proof of) \cite[Theorem 3.3]{katz_padic} as well as \cite[pp.\ 112--113]{katz_padic} for various facts regarding the general modular definition of the Frobenius operator $V$: Let $r_1\in O$ with $v_p(r_1)<1/(p+1)$ and let $((\sh{E}_1)_{/R},\omega_1,\alpha_N,Y_1)$ be an $r_1$-test object. Let $\sh{E}'_1 = \sh{E}_1/C$ with $C$ the canonical subgroup and $\pi : \mor{\sh{E}_1}{\sh{E}'_1}$ be the quotient map. Let $\omega'_1$ be a non-vanishing invariant differential on $\sh{E}'_1$. Then $\breve{\pi}^{\ast}(\omega_1) = \lambda \omega'_1$ for some $\lambda = \lambda\left(\sh{E}_1,\omega_1 \right) \in R$. This $\lambda\left(\sh{E}_1,\omega_1 \right)$ and $E_{p-1}\left(\sh{E}_1,\omega_1 \right)$ are both lifts of the Hasse invariant of $(\sh{E}_1,\omega_1)$, and so the latter differs from the first by a unit. For ease of writing, and as in the proof of \cite[Theorem 3.3]{katz_padic}, we denote this unit by
$$
\frac{E_{p-1}\left(\sh{E}_1, \omega_1 \right)}{\lambda\left(\sh{E}_1, \omega_1 \right)} .
$$

With $\alpha'_N = \pi(\alpha_N)$, the $r_1$-test object $((\sh{E}_1)_{/R},\omega_1,\alpha_N,Y_1)$ gives rise to an $r_1^p$-test object $((\sh{E}'_1)_{/R},\omega'_1,\alpha'_N,Y_1')$, and for an element $g\in M_k(N,O,r_1^p)$, the modular definition of $V(g)$ now reads:
$$
E_{p-1}^k \cdot V(g)(\sh{E}_1,\omega_1,\alpha_N,Y_1) = \left(\frac{E_{p-1}(\sh{E}_1, \omega_1)}{\lambda(\sh{E}_1, \omega_1)} \right)^k\cdot g(\sh{E}'_1,\omega'_1,\alpha'_N,Y'_1) .
$$

Using $E_{p-1} \cdot Y_1 = r_1$ and specializing to $g=f\in M_k(N,O,r)$ as well as $r_1 = 1$, we then have:
\begin{eqnarray*}
V(f)(\sh{E}_1,\omega_1,\alpha_N,Y_1) &=& Y_1^k \cdot \left(\frac{E_{p-1}(\sh{E}_1, \omega_1)}{\lambda(\sh{E}_1, \omega_1)} \right)^k\cdot f(\sh{E}'_1,\omega'_1,\alpha'_N,Y'_1) \\
&=& Y_1^k \cdot \left(\frac{E_{p-1}(\sh{E}_1, \omega_1)}{\lambda(\sh{E}_1, \omega_1)} \right)^k\cdot f(\sh{E}'_1,\omega'_1,\alpha'_N,C')
\end{eqnarray*}
($C'$: canonical subgroup of $\sh{E}'$.) We use this formula as inspiration to {\it define} an element $F\in M_k(N,O,r)$ by
$$
F(\sh{E},\omega,\alpha_N,Y) = Y^k \cdot \left(\frac{E_{p-1}(\sh{E}, \omega)}{\lambda(\sh{E}, \omega)} \right)^k\cdot f(\sh{E}',\omega',\alpha'_N,C')
$$
for an $r$-test object $(\sh{E},\omega,\alpha_N,Y)$ (and $\sh{E}'$ etc.\ defined as above.) We now compare $F\in M_k(N,O,r)$ as embedded into $M_k(N,O,1)$ with $V(f)$. Let us denote by $\iota$ the embedding $M_k(N,O,r) \hookrightarrow M_k(N,O,1)$ as described before the statement of the lemma. Then, for a $1$-test object $(\sh{E}_1,\omega_1,\alpha_N,Y_1)$ we find
\begin{eqnarray*}
&&(\iota F) (\sh{E}_1,\omega_1,\alpha_N,Y_1) = F (\sh{E}_1,\omega_1,\alpha_N,rY_1) \\
&& = r^k \cdot Y_1^k \cdot \left(\frac{E_{p-1}(\sh{E}_1, \omega_1)}{\lambda(\sh{E}_1, \omega_1)} \right)^k\cdot f(\sh{E}'_1,\omega'_1,\alpha'_N,C') = r^k V(f) (\sh{E}_1,\omega_1,\alpha_N,Y_1) ,
\end{eqnarray*}
and we are done.
\end{proof}

\begin{remark} In \cite[p.\ 38, Remark 2]{gouvea_p-adic_modular_forms} it is stated that if $f$ is classical of weight $k$ with coefficients in $O$ and level $\Gamma_1(N)$ with $N\ge 3$ prime to $p$, then if $v_p(r)<p/(p+1)$, both $f$ and $V(f)$ embed into $M_k(N,O,r)$, the argument being that both $f$ and $V(f)$ are classical on $\Gamma_1(N)\cap \Gamma_0(p)$. It is true that these forms are classical of course, but the claim is nonetheless false in general, the problem being that $V(f)$ might not be defined over $O$. In down to earth terms this is because the Fourier expansion of $V(f)$ around some cusps might not be integral anymore. However, we will have $p^k V(f) \in M_k(N,O,r)$ as the above lemma shows.
\end{remark}

\begin{lem}\label{suppintlemma} Let $n\in\N$ and let $f \in M_{n(p-1)}(O)$ be a classical modular form of level $1$, weight $n(p-1)$, and coefficients in $O$.
\begin{enumerate}[(i)]
\item
$$
\frac{p^{n(p-1)} V(f)}{E_{p-1}^n} \in M_0\left(O, \geq \frac{p-1}{p} \right).
$$
\item If, furthermore, we have
$$
\frac{f}{E_{p-1}^{n}} \in M_0\left(O, \geq \frac{1}{p+1} \right),
$$
then
$$
\frac{p^{n(p-1)} V(f)}{E_{p-1}^{n}} \in M_0\left(O, \geq \frac{p}{p+1} \right).
$$
\end{enumerate}
\end{lem}
\begin{proof} Put $k:=n(p-1)$.
\smallskip

\noindent (i) Again referring back to Proposition \ref{prop:overconv_greater_than_or_equal_rho}, it suffices to consider a (sufficiently large) finite extension $K'/K$ with ring of integers $O'$, an $r\in O'$ with $v_p(r) < \frac{p-1}{p}$, and then show that $p^{n(p-1)} V(f)/E_{p-1}^n\in M_0(O',r)$. Again, choosing an auxiliary level $N\ge 3$ prime to $p$ -- for instance $N=3$ -- and embedding $f$ into $M_k(N,O',r)$, it suffices to show that $p^{n(p-1)} V(f)/E_{p-1}^n\in M_0(N,O',r)$ for such $O',r$ (recall Remark \ref{rem:higher_level_=>rho}.)

Let then $N$, $O'$, $r$ be like that. As $v_p(r) < \frac{p-1}{p} < \frac{p}{p+1}$, Lemma \ref{lem:r^kV(f)} applies and shows that $r^k V(f) \in M_k(N,O',r)$. Because of the relation $E_{p-1} \cdot Y = r$ for $r$-test objects $(\sh{E}_{/R},\omega,\alpha_N,Y)$ we then find that
$$
r^{k+n} \cdot \frac{V(f)}{E_{p-1}^n} \in M_0(N,O',r) .
$$

But we have $v_p(p^k) = k = n(p-1) > v_p(r^{k+n}) = np \cdot v_p(r)$ because $v_p(r) < \frac{p-1}{p}$.
\smallskip

\noindent (ii) We retain the setup and the discussion in part (i) and assume additionally that
$$
\frac{f}{E_{p-1}^{n}} \in M_{0}\left(O, \geq \frac{1}{p+1}\right) .
$$

We then have a Katz expansion
$$
\frac{f}{E_{p-1}^{n}} = \sum_{i = 0}^{n} \frac{b_i}{E_{p-1}^i}
$$
where the $b_i\in B_i(O)$ satisfy $v_p(b_i) \ge \frac{i}{p+1}$, cf.\ Proposition \ref{prop:katz_exp_classical_form}. So,
$$
f = \sum_{i = 0}^{n} E_{p-1}^{n-i} b_i .
$$

As in part (i) we choose an auxiliary $N=3$ as well as a finite extension $K'$ of $K$ with ring of integers $O'$. It suffices to show that $p^k \frac{V(f)}{E_{p-1}^n} \in M_0(N,O',r)$ if $r\in O'$ has $v_p(r)<p/(p+1)$.

Assuming initially $v_p(r)<(p-1)/p$, we have from the discussion in part (i) as well as the proof of Lemma \ref{lem:r^kV(f)} that $r^{k+n} \frac{V(f)}{E_{p-1}^n} = r^{np} \frac{V(f)}{E_{p-1}^n} \in M_0(N,O',r)$ with the following value on an $r$-test object $(\sh{E}_{/R},\omega,\alpha_N,Y)$:
\begin{eqnarray*}
&& r^{np} \frac{V(f)}{E_{p-1}^n} (\sh{E}_{/R},\omega,\alpha_N,Y) = Y^{k+n} \cdot \left(\frac{E_{p-1}(\sh{E}, \omega)}{\lambda(\sh{E}, \omega)} \right)^k\cdot f(\sh{E}',\omega',\alpha'_N,C') \\
&=& \left(\frac{E_{p-1}(\sh{E}, \omega)}{\lambda(\sh{E}, \omega)} \right)^k\cdot Y^{np} \sum_{i = 0}^{n} E_{p-1}^{n-i}(\sh{E}', \omega') b_i(\sh{E}', \omega') .
\end{eqnarray*}
with notation as in the proof of Lemma \ref{lem:r^kV(f)}. As $v_p(r)<(p-1)/p$ the element $\frac{p^{p-1}}{r^p}$ has positive valuation and hence so does $\frac{p^k}{r^{np}} = \left( \frac{p^{p-1}}{r^p} \right)^n$, and we see that
$$
p^k \frac{V(f)}{E_{p-1}^n} = \left( \frac{p^k}{r^{np}}\right) \cdot r^{np} \frac{V(f)}{E_{p-1}^n} \in M_0(N,O',r)
$$
with the value of this function on the $r$-test object $(\sh{E}_{/R},\omega,\alpha_N,Y)$ equalling
\begin{eqnarray*}
&& \left(\frac{E_{p-1}(\sh{E}, \omega)}{\lambda(\sh{E}, \omega)} \right)^k \sum_{i = 0}^{n} \left( \frac{p^{p-1}}{r^p} \right)^n Y^{np} \cdot E_{p-1}^{n-i}(\sh{E}', \omega') b_i(\sh{E}', \omega') \\
&=& \left(\frac{E_{p-1}(\sh{E}, \omega)}{\lambda(\sh{E}, \omega)} \right)^k \sum_{i = 0}^{n} \left( \frac{p^{p-1}}{r^p} Y^p E_{p-1}(\sh{E}', \omega') \right)^{n-i} \cdot \left( \left( \frac{p^{p-1}}{r^p} \right)^i b_i(\sh{E}', \omega') \right) Y^{pi} .
\end{eqnarray*}

If we now relax our assumption on $r$ to just $v_p(r) < p/(p+1)$, it is clear that we have $p^k \frac{V(f)}{E_{p-1}^n} \in M_0(N,K',r) = M_0(N,O',r) \otimes K'$. We show that in fact $p^k \frac{V(f)}{E_{p-1}^n} \in M_0(N,O',r)$ by showing that both of the factors
$$
\frac{p^{p-1}}{r^p} Y^p E_{p-1}(\sh{E}', \omega') ~\mbox{and}~ \left(\frac{p^{p-1}}{r^p} \right)^i b_i(\sh{E}', \omega')
$$
in the sum above are in fact elements of $R$ under our assumptions. For the last of these, this follows readily: we have
$$
v_p\left( \left(\frac{p^{p-1}}{r^p} \right)^i \right) > i\cdot \left((p-1) - \frac{p^2}{p+1} \right) = -\frac{i}{p+1} ,
$$
and the claim follows as we have $v_p(b_i) \ge \frac{i}{p+1}$.

To see the claim for the first factor, put $r_1 := \frac{p}{r}$. Since the canonical subgroup of $\sh{E} \pmod{r_1R}$ becomes the kernel of the Frobenius map $\sh{E}\otimes R/r_1 R \rightarrow \left(\sh{E}\otimes R/r_1R \right)^{(p)}$, we have
$$
E_{p-1}\left(\sh{E}', \omega'\right) = E_{p-1}^p\left(\sh{E}, \omega\right) + r_1 h
$$
for some $h \in R$, where we have chosen and fixed $\omega'$ to be a nowhere vanishing one-form on $\sh{E}'$ that is congruent to $\omega^{(p)}$ on $\sh{E}^{(p)}$ modulo $r_1R$. Confer the discussion pp.\ 118--121 of \cite{katz_padic}, and in particular equation (3.9.1) (notice that we have $v_p(r_1)>1/(p+1)$ so that $t_{\rm can} \in r_1 R$ in the terminology on p.\ 120 of \cite{katz_padic}.)

Using the relation $Y\cdot E_{p-1}(\sh{E}, \omega) = r$ we then find
$$
\frac{p^{p-1}}{r^p} Y^p E_{p-1}(\sh{E}', \omega') = \frac{p^{p-1}}{r^p} (r^p + \frac{p}{r}h Y^p) = p^{p-1} + \frac{p^p}{r^{p+1}}hY^p ,
$$
an element of $R$ as
$$
v_p\left(p^p / r^{p+1}\right) = p - (p+1)v_p(r) > 0 .
$$
\end{proof}

\subsection{Overconvergence of \texorpdfstring{$e_n^{\ast}$}{en*}}\label{subsec:overconvergence_en*} A central element of our study of $\es_n$ is the following statement that is an immediate consequence of a theorem of Serre, cf.\ \cite[Th\'{e}or\`{e}me 8]{serre_zeta}.

\begin{thm}(Serre)\label{projpoly} Let $k\in (p-1)\N$.

There exists a polynomial $H_k \in U\Z_p [U, \{T_\ell \mid \ell \not = p\}]$ such that,
\begin{enumerate}[(i)]
\item $H_k(E_k^\ast) = \Es_k$, and
\item $\lim_{i \rightarrow \infty} H_k^i(f) = 0$ for every cuspidal Serre $p$-adic modular form $f$ of weight $k$.
\end{enumerate}
\end{thm}

Let us briefly explain why Theorem \ref{projpoly} follows from Serre's theorem, Th\'{e}or\`{e}me 8 of \cite{serre_zeta}. A consequence of the latter theorem is the existence of a polynomial $H$ in $U$ and the $T_{\ell}$, $\ell\neq p$, with integral coefficients such that for every $k\in\N$ divisible by $p-1$ we have $H(E_k^{\ast}) = c(k) E_k^{\ast}$ with $c(k)\in \Z_p^{\times}$, and such that $\lim_{i \rightarrow \infty} H^i(f) = 0$ for every cuspidal Serre $p$-adic modular form $f$ of weight $k$.

Thus, for a fixed $k\in \N$ with $k\equiv 0 \pmod{p-1}$ we can take the $H_k$ in Theorem \ref{projpoly} to be $H_k := c(k)^{-1} H$ with the $H$ from \cite[Th\'{e}or\`{e}me 8]{serre_zeta}, where the only thing that is not immediately clear is the fact that we can take $H \in U\Z_p [U, \{T_\ell: \ell \not = p\}]$. However, that fact follows by an inspection of the proof, cf.\ the end of the proof of \cite[Th\'{e}or\`{e}me 8]{serre_zeta} on p.\ 219 of \cite{serre_zeta}.

\begin{remark} Inspection of the proof of \cite[Th\'{e}or\`{e}me 8]{serre_zeta} gives an algorithm for determining a polynomial $H_k$ as above. Serre notes, cf.\ loc.\ cit.\ p.\ 219, that we can take $H_k = U$ if $p \in \{2,3,5,7\}$, and $H_k = 11U(U+5)$ for $p=13$.

For $p=17$, Serre suggests $H_k = \frac{1}{2^{k-1} + 6} U(T_2 + 5)$. However, $H_k = 14U(U+10)$ also works.

We have investigated the situation numerically for all $p \leq 500$ and have found that it is possible for those primes to choose $H_k$ to be a polynomial in the $U$ operator alone. Whether this continues to be the case for higher primes, we do not know.
\end{remark}

\begin{lem}\label{lem:limit_H_n(p-1)^i_on_E(p-1)^n} Let $n \in \N$ and $H_{n(p-1)}$ be a polynomial as in Theorem \ref{projpoly}. Then
$$
\lim_{i \rightarrow \infty} H_{n(p-1)}^i(E_{p-1}^n) = \Es_{n(p-1)}
$$
in $q$-expansions.
\end{lem}

\begin{proof} Let $H = H_{n(p-1)}$. The Serre $p$-adic modular form $f = E_{p-1}^n - \Es_{n(p-1)}$ is cuspidal of weight $n(p-1)$. By Theorem \ref{projpoly}(ii), we have $H^i(f) \rightarrow 0$, and by Theorem \ref{projpoly}(i), $H^i(f) = H^i(E_{p-1}^n) - \Es_{n(p-1)}$, whence the result.
\end{proof}

\begin{dfn}\label{def:HH_n} For each $n \in \N$, define the operator
$$
\HH_n(f) : M_0(K, 1) \rightarrow M_0(K, 1)
$$
by
$$
\HH_n(f) := \frac{H_{n(p-1)}(f\cdot E_{p-1}^n)}{E_{p-1}^n}
$$
where $H_{n(p-1)}$ is a polynomial as in Theorem \ref{projpoly}.
\end{dfn}

We will need to study the iterations of this twisted operator, $\HH_n$, and specifically its value on $1$. From its definition, the definition of the operators $\U_n$ and $\T_{\ell,n}$ ($\ell$ prime $\neq p$) (Definition \ref{def:twisted_operators}), as well as the fact that the operators $U$ and $T_{\ell}$ commute, it follows that for a given $i\in\N$ we can write $\HH_n^i = \sum_{j = 1}^d F_j \U_n^j$ with the $F_j$ certain polynomials in $\Z_p [\{\T_{\ell,n} \mid \ell \not = p\}]$. To study $\HH_n^i(1)$ we first study iterates of $\U_n$ as well as polynomials in the $\T_{\ell,n}$. The following congruence, a simple consequence of the von Staudt--Clausen theorem, will be crucial for that purpose.

\begin{lem}\label{elem_cong_eis_p2} For $n \in \N$ we have
$$
e_n \equiv \es_n \equiv 1 \pmod{p^2}.
$$
\end{lem}
\begin{proof} Recall our assumption that $p\ge 5$. By definition,
$$
E_{n(p-1)}^\ast = \frac{E_{n(p-1)} - p^{n(p-1)-1}V(E_{n(p-1)})}{1 - p^{n(p-1)-1}}
$$
and therefore $E_{n(p-1)}^\ast \equiv E_{n(p-1)} \pmod{p^2}$. So it is enough to show that
$$
E_{n(p-1)} \equiv E_{p-1}^n \pmod{p^2}.
$$

By the von Staudt--Clausen theorem, we have
$$
B_{p-1}, B_{n(p-1)} \equiv \frac{-1}{p} \pmod{\Z_{p}}
$$
which means that $B_{p-1} = \frac{-u}{p}$ and $B_{n(p-1)} = \frac{-v}{p}$ for some $1$-units $u,v \in \Z_{p}$. Hence,
$$
\frac{1}{B_{p-1}} = -pu^{-1} \equiv -pv^{-1} = \frac{1}{B_{n(p-1)}} \pmod{p^2 \Z_p}
$$
as $u^{-1}, v^{-1}$ are again $1$-units and hence $\equiv 1 \pmod{p\Z_{p}}$. Now, computing modulo $p^2 \Z_{p}$ we find
\begin{eqnarray*}
E_{n(p-1)} &=& 1 - \frac{2n(p-1)}{B_{n(p-1)}}\sum_{i \geq 1} \sigma_{n(p-1)-1}(i)q^i \equiv 1 - \frac{2n(p-1)}{B_{p-1}}\sum_{i \geq 1} \sigma_{n(p-1)-1}(i)q^i \\
&\equiv& 1 - \frac{2n(p-1)}{B_{p-1}}\sum_{i \geq 1}\sigma_{p-2}(i)q^i \equiv \left(1 - \frac{2(p-1)}{B_{p-1}}\sum_{i \geq 1}\sigma_{p-2}(i)q^i\right)^n \\
&=& E_{p-1}^n \pmod{p^2 \Z_p} \\
\end{eqnarray*}
using $\frac{1}{B_{p-1}} \equiv 0 \pmod{p\Z_{p}}$ as well as
$$
\sigma_{n(p-1)-1}(i) \equiv \sigma_{p - 2}(i) \pmod{p\Z_{p}}
$$
by Fermat's little theorem.
\end{proof}

As will be seen, the congruence is used in the proof of the following lemma.

\begin{lem}\label{twisted_operators_on_1} Let $n, i \in \N$ and $F \in \Z_{p}[\{\T_{\ell,n}\mid \ell \not = p\}]$. Let $c \in \Z_{p}$ be defined by $F(\es_n) = c\cdot\es_n$ (recall Proposition \ref{prop:es_n_eigenfct}.) Then:
\begin{enumerate}[(i)]
\item
$$
\U_n^i(1) \in \U_n(1) + M_0\left(\Z_p, \geq \frac{p}{p+1}\right),
$$
\item
$$
F(\U_n(1)) \in c\cdot \U_n(1) + M_0\left(\Z_p, \geq \frac{p}{p+1}\right),
$$
\item
$$
F(\U_n^i(1)) \in c\cdot \U_n(1) + M_0\left(\Z_p, \geq \frac{p}{p+1}\right).
$$
\end{enumerate}
\end{lem}

\begin{proof}
\noindent (i) We show this by induction on $i$. For $i = 1$, there is nothing to show. Suppose $i \geq 1$ and
$$
\U_n^i (1) \in \U_n(1) + M_0\left(\Z_p, \geq \frac{p}{p+1}\right).
$$

In particular, by the Integrality Lemma, Lemma \ref{integrality_lemma}, we have
$$
\U_n^i(1) \in \frac{1}{p}M_0\left(\Z_p, \geq \frac{p}{p+1}\right) ,
$$
and hence a Katz expansion
$$
\U_n^i(1) = \sum_{t \geq 0} \frac{b_t}{E_{p-1}^t}
$$
with $v_p(b_t) \geq \frac{pt}{p+1} - 1$ for all $t \geq 0$.

By Lemma \ref{elem_cong_eis_p2} we have $\es_n \equiv 1 \pmod{p^2}$ (in $q$-expansions), or, equivalently, $E_{n(p-1)}^{\ast} \equiv E_{p-1}^n \pmod{p^2}$ since $E_{p-1} \equiv 1 \pmod{p}$. Since the classical $U$ operator preserves congruences between $q$-expansions, using Proposition \ref{prop:es_n_eigenfct} as well as the definition of $\U_n$ we find $\es_n = \U_n(\es_n) \equiv \U_n(1) \pmod{p^2}$. Similarly, the congruence $e_n\equiv 1 \pmod{p^2}$ gives $\U_n(1) \equiv 1 \pmod{p^2}$. By induction on $i$ we obtain:
$$
\U_n^i(1) - 1 \equiv \U_n^i(\es_n) - \es_n \equiv 0 \pmod{p^2} .
$$

Applying Proposition \ref{prop:congr_in_q-exp_implies_congr_for Katz_exp}, we find in particular that $v_p(b_0-1) \ge 2 > 1$, as well as
$$
v_p(b_t) \geq 2 \geq 1 + \frac{t}{p+1},
$$
for $t=1,2$.

Furthermore, if $t \geq 3$ then $t \geq \frac{2(p+1)}{p-1}$ (as $p\ge 5$), and hence
$$
v_p(b_t) - \left( 1 + \frac{t}{p+1}\right) \geq \frac{pt}{p+1}-1 -1 - \frac{t}{p+1} = \frac{(p-1)t - 2(p+1)}{p+1} \geq 0.
$$

We conclude that
$$
\U_n^i(1) - 1 \in pM_0\left(\Z_p, \geq \frac{1}{p+1}\right).
$$

Applying the $\U_n$ operator and using Lemma \ref{integrality_lemma} again, we obtain
$$
\U_n^{i+1}(1) - \U_n(1) \in M_0\left(\Z_p, \geq \frac{p}{p+1}\right),
$$
and we are done.

\noindent (ii) Let $F \in \Z_{p}[\{\T_\ell: \ell \not = p\}]$ and let $c \in \Z_{p}$ such that $F(\es_n) = c\cdot\es_n$. By the Integrality Lemma (Lemma \ref{integrality_lemma}),
$$
F(1) \in M_0\left(\Z_p, \geq \frac{p}{p+1}\right)
$$
and therefore for the Katz expansion
$$
F(1) = \sum_{t \geq 0} \frac{b_t}{E_{p-1}^t}
$$
we have $v_p(b_t) \geq \frac{pt}{p+1}$. By Lemma \ref{elem_cong_eis_p2} we have
$$
F(1) \equiv F(\es_n) \equiv c \pmod{p^2} ,
$$
and hence $v_p(b_0 - c) \geq 2 > 1$ and $v_p(b_1) \geq 2 > 1 + \frac{1}{p+1}$, again by Proposition \ref{prop:congr_in_q-exp_implies_congr_for Katz_exp}. Furthermore, if $t \ge 2$ then $t > \frac{p+1}{p-1}$ and hence
$$
v_p(b_t) - \left(1 + \frac{t}{p+1}\right) \geq \frac{pt}{p+1} - \left(1 + \frac{t}{p+1}\right) \ge \frac{(p-1)t-(p+1)}{p+1} > 0.
$$

We get
$$
F(1) - c \in pM_0\left(\Z_p, \geq \frac{1}{p+1}\right).
$$

Applying $\U_n$ and once more using Lemma \ref{integrality_lemma}, we get
$$
\U_n(F(1)) \in c\cdot\U_n(1) + M_0\left(\Z_p, \geq \frac{p}{p+1}\right).
$$

As the operators $\U_n$ and the $\T_{\ell,n}$ commute, so do $\U_n$ and $F$, and we are done.

\noindent (iii) By part (i),
$$
\U_n^i(1) \in \U_n(1) + M_0\left(\Z_p, \geq \frac{p}{p+1}\right) ,
$$
and so by Lemma \ref{integrality_lemma},
$$
F(\U_n^i(1)) \in F(\U_n(1)) + M_0\left(\Z_p, \geq \frac{p}{p+1}\right).
$$

By part (ii), we then obtain
$$
F(\U_n^i(1)) \in c\cdot\U_n(1) + M_0\left(\Z_p, \geq \frac{p}{p+1}\right).
$$
\end{proof}

\begin{prop}\label{ens_oc} For any $n,i \in \N$, we have
$$
\HH_n^i(1) \in \U_n(1) + M_0\left(\Z_p, \geq \frac{p}{p+1}\right) ,
$$
and consequently,
$$
e_n^\ast \in \U_n(1) + M_0\left(\Z_p, \geq\frac{p}{p+1}\right).
$$
\end{prop}

\begin{proof} As we argued immediately after the definition of $\HH_n$ (cf.\ Definition \ref{def:HH_n} as well as Theorem \ref{projpoly}), we may write
$$
\HH_n^i = \sum_{j = 1}^d F_j \U_n^j
$$
for some $d$ and with certain polynomials $F_j \in \Z_p [\{\T_{\ell,n} \mid \ell \not = p\}]$. For each $j \in \{1, \ldots, d\}$, let $c_j \in \Z_{p}$ such that $F_j(\es_n) = c_j \es_n$. Put $c := \sum_{j=1}^d c_j$. By Lemma \ref{twisted_operators_on_1}, we have
$$
F_j(\U_n^j (1)) \in c_j\cdot \U_n(1) + M_0\left(\Z_p, \geq \frac{p}{p+1}\right)
$$
for $j=1,\ldots,d$ and hence
$$
\HH_n^i(1) \in c\cdot\U_n(1) + M_0\left(\Z_p, \geq \frac{p}{p+1}\right).
$$

But, again by the definition of $\HH_n$ and the fact that $\U_n \es_n = \es_n$ (Proposition \ref{prop:es_n_eigenfct}), we have
$$
\es_n = \HH_n(\es_n) = c\cdot\es_n
$$
and therefore $c = 1$.

By the definition of $\HH_n$ and Lemma \ref{lem:limit_H_n(p-1)^i_on_E(p-1)^n} we have
$$
\es_n = \lim_{i \rightarrow \infty} \HH_n^i(1)
$$
in $q$-expansions, and so Corollary \ref{cor:conv_in_qexp} implies
$$
e_n^\ast \in \U_n(1) + M_0\left(\Z_p, \geq\frac{p}{p+1}\right).
$$
\end{proof}

\subsection{Overconvergence of \texorpdfstring{$e_n$}{en}}\label{subsec:overconvergence_en} In this subsection and further down below the following definition will be convenient.

\begin{dfn}\label{def:T_{p,n}} For $n\in\N$ define
$$
\T_{p,n}(1) = \frac{T_p(E_{p-1}^n)}{E_{p-1}^n}
$$
where $T_p$ is the $p$th Hecke operator acting on classical forms of level $1$ and weight $n(p-1)$. Thus, in concrete terms,
$$
\T_{p,n}(1) = \frac{U(E_{p-1}^n) + p^{n(p-1)-1}V(E_{p-1}^n)}{E_{p-1}^n} .
$$
\end{dfn}

\begin{prop}\label{comparing_Tp_en} For any $n \in \N$ we have
$$
e_n \in \frac{1}{p}M_0\left(\Z_p, \geq \frac{p-1}{p} \right).
$$
and
$$
e_n \in \T_{p,n}(1) + M_0\left(\Z_p, \geq \frac{p}{p+1} \right).
$$
\end{prop}

\begin{proof} We showed in Proposition \ref{ens_oc} that
$$
\es_n \in \U_n(1) + M_0\left(\Z_p, \geq \frac{p}{p+1} \right).
$$

Unpacking this, it means that
$$
\frac{\Es_{n(p-1)}}{E_{p-1}^n} = \frac{E_{n(p-1)} - p^{n(p-1)-1}V(E_{n(p-1)})}{(1 - p^{n(p-1)-1})E_{p-1}^n} \in \U_n(1) + M_0\left(\Z_p, \geq \frac{p}{p+1} \right).
$$

Since $p\U_n(1) \in M_0\left(\Z_p, \geq \frac{p}{p+1}\right)$ by Lemma \ref{integrality_lemma}, we get
$$
\frac{E_{n(p-1)} - p^{n(p-1)-1}V(E_{n(p-1)})}{E_{p-1}^n} \in \U_n(1) + M_0\left(\Z_p, \geq \frac{p}{p+1} \right)
$$
and hence
$$
e_n \in \U_n(1) + \frac{p^{n(p-1)-1}V(E_{n(p-1)})}{E_{p-1}^n} + M_0\left(\Z_p, \geq \frac{p}{p+1} \right).
$$

Using Lemma \ref{integrality_lemma} and Lemma \ref{suppintlemma} (i), we obtain
$$
e_n \in \frac{1}{p}M_0\left(\Z_p, \geq \frac{p-1}{p} \right).
$$

The Katz expansion of $e_n$ is therefore
$$
e_n = 1 + \sum_{t \geq 1} \frac{b_t}{E_{p-1}^t}
$$
where $v_p(b_t) \ge \frac{t(p-1)}{p} - 1$.

Now, by Lemma \ref{elem_cong_eis_p2} we have $e_n \equiv 1 \pmod{p^2}$. If we combine this with Proposition \ref{prop:congr_in_q-exp_implies_congr_for Katz_exp}, we see that
$$
v_p(b_t) \geq 2 \geq 1 + \frac{t}{p+1}
$$
for $t=1,2,3$. On the other hand, if $t\ge 4$ then one finds that $\frac{t(p-1)}{p} - 1 \ge 1 + \frac{t}{p+1}$ because $p\ge 5$. We conclude that $v_p(b_t) \ge 1 + \frac{t}{p+1}$ for all $t\ge 1$ and thus
$$
\frac{e_n-1}{p} \in M_0\left(\Z_p, \ge \frac{1}{p+1}\right).
$$

Applying Lemma \ref{suppintlemma} (ii) to $f := \frac{E_{n(p-1)} - E_{p-1}^n}{p}$ we then obtain
$$
\frac{p^{n(p-1)-1}V(E_{n(p-1)} - E_{p-1}^n)}{E_{p-1}^n} \in M_0\left(\Z_p, \geq \frac{p}{p+1}\right).
$$

Returning to
$$
e_n \in \U_n(1) + \frac{p^{n(p-1)-1}V(E_{n(p-1)})}{E_{p-1}^n} + M_0\left(\Z_p, \geq \frac{p}{p+1} \right),
$$
we now have
$$
e_n \in \frac{U(E_{p-1}^n)}{E_{p-1}^n} + \frac{p^{n(p-1)-1}V(E_{n(p-1)})}{E_{p-1}^n} + M_0\left(\Z_p, \geq \frac{p}{p+1} \right)
$$
$$
= \T_{p,n}(1) + \frac{p^{n(p-1)-1}V(E_{n(p-1)} - E_{p-1}^n)}{E_{p-1}^n} + M_0\left(\Z_p, \geq \frac{p}{p+1} \right)
$$
$$
= \T_{p,n}(1) + M_0\left(\Z_p, \geq \frac{p}{p+1} \right).
$$
\end{proof}

\section{Proofs of the theorems}\label{sec:proofs}

\subsection{Proof of Theorem \ref{thm:es_n_oc}}\label{subsec:proof_e_n_es_n}

\begin{proof}[Proof of Theorem \ref{thm:es_n_oc}] Combining the last statement of Proposition \ref{ens_oc} with Lem\-ma \ref{integrality_lemma} (i), we obtain
$$
\es_n \in \frac{1}{p} M_0\left(\Z_p,\ge \frac{p}{p+1} \right) .
$$

Lemma \ref{suppintlemma} applied to $f=E_{p-1}^n$ shows that
$$
\frac{p^{n(p-1)}V(E_{p-1}^n)}{E_{p-1}^n} \in M_0\left(\Z_p, \ge \frac{p}{p+1} \right) .
$$

We also have
$$
p\cdot \frac{U(E_{p-1}^n)}{E_{p-1}^n} = p\cdot \U_n(1) \in M_0\left(\Z_p, \ge \frac{p}{p+1} \right)
$$
by Lemma \ref{integrality_lemma}. Thus, the definition of $\T_{p,n}(1)$:
$$
\T_{p,n}(1) = \frac{U(E_{p-1}^n) + p^{n(p-1)-1}V(E_{p-1}^n)}{E_{p-1}^n}
$$
shows that
$$
p\cdot \T_{p,n}(1) \in M_0\left(\Z_p, \ge \frac{p}{p+1} \right)
$$
and then by Proposition \ref{comparing_Tp_en} we have
$$
e_n \in \frac{1}{p} M_0\left(\Z_p,\ge \frac{p}{p+1} \right) .
$$

Consider the Katz expansion $\sum_{t\ge 0} \frac{b_t}{E_{p-1}^t}$ of either $e_n$ or $\es_n$. We have $b_0=1$. Because of the congruences $e_n \equiv \es_n \equiv 1\pmod{p^2}$ from Lemma \ref{elem_cong_eis_p2} combined with Proposition \ref{prop:congr_in_q-exp_implies_congr_for Katz_exp}, we have $v_p(b_t) \ge 2$ for $t\ge 1$. We see that
$$
e_n,\es_n \in M_0\left( \Z_p,\ge \rho \right)
$$
whenever $\rho$ is picked such that either $\rho t \le \frac{pt}{p+1} - 1$ or $\rho t \le 2$ for all $t\ge 1$. If we put
$$
\rho = \frac{2}{3} \cdot \frac{p}{p+1} ,
$$
we will have $\rho t \le \frac{pt}{p+1} - 1$ when $t\ge 4$, and $\rho t \le 2$ when $t\le 3$. The second statement of the theorem follows.
\end{proof}

\subsection{Proof of Theorem \ref{thm:V(E)/E_oc}} We now prove Theorem \ref{thm:V(E)/E_oc}. We first notice the following immediate corollary to Theorem \ref{thm:es_n_oc}.

\begin{cor}\label{cor:en_esn_1-units} Let $u$ with $0 \le u < \frac{2}{3}\cdot \frac{p}{p+1}$ be given. Then for any sufficiently large finite extension $K/\Q_p$ with ring of integers $O$ we have that $e_n$ and $\es_n$ both are $1$-units in $M_0(O,r)$ whenever $r\in O$ with $v_p(r) \le u$.
\end{cor}

\begin{proof} Put $\rho := \frac{2}{3}\cdot \frac{p}{p+1}$. The proof of the corollary is the same for $e_n$ and $\es_n$, so let us just consider $e_n$. As $e_n \in M_0\left( \Z_p,\ge \frac{2}{3} \cdot \frac{p}{p+1} \right)$ by Theorem \ref{thm:es_n_oc}, the Katz expansion of $e_n$ has form
$$
e_n = 1 + \sum_{i\ge 1} \frac{b_i}{E_{p-1}^i}
$$
where the $b_i \in B_i(\Z_p)$ satisfy $v_p(b_i) \ge \rho \cdot i$ for all $i$.

Choose $K$ and $O$ sufficiently large so that we have an element $a\in O$ with $0 < v_p(a) \le \frac{1}{2} \cdot (\rho - u)$. Define $b_i' := a^{-1} b_i$ for $i\ge 1$. If then $r\in O$ with $v_p(r) \le u$ we find
$$
v_p(b_i') - iv_p(r) \ge (\rho - u)(i-\frac{1}{2})
$$
for $i\ge 1$ which shows that $v_p(b_i') - iv_p(r) \ge 0$ for $i\ge 1$ as well as $v_p(b_i') - iv_p(r) \rightarrow \infty$ for $i\rightarrow \infty$. Thus,
$$
f := \sum_{i\ge 1} \frac{b_i'}{E_{p-1}^i}
$$
is seen to be the Katz expansion of an element of $M_0(O,r)$, and as we have $e_n = 1 + a\cdot f$ with $v_p(a) > 0$, we are done.
\end{proof}

\begin{proof}[Proof of Theorem \ref{thm:V(E)/E_oc}] Let $k\ge 4$ be an integer divisible by $p-1$ and put $n:=\frac{k}{p-1}$.

By Corollary \ref{cor:en_esn_1-units} we see that there exists a sequence of finite extensions $K_t/\Q_p$ with rings of integers $O_t$ as well as elements $s_t\in O_t$ such that $e_n$ is a $1$-unit in $M_0(O_t,s_t)$ for each $t$, and such that $v_p(s_t)$ converges to $\frac{2}{3} \cdot \frac{p}{p+1}$ from below. Enlarging the $K_t$ if necessary we may assume that each $s_t$ is a $p$th power in $O_t$, say $s_t = r_t^p$. We then have:
\begin{itemize}
\item $e_n$ is a $1$-unit in $M_0(O_t,r_t^p)$ for each $t$,
\item the values $v_p(r_t)$ converge to $\frac{2}{3} \cdot \frac{1}{p+1}$ from below.
\end{itemize}

Obviously, we may further assume that $(O_t,r_t)=(\Z_p,1)$ for one $t$.

By the well-known property of the Frobenius operator $V$ that it maps $M_0(O_t,r_t^p)$ to $M_0(O_t,r_t)$, cf.\ \cite[Theorem 3]{katz_padic}, it follows that $V(e_n)$ is a $1$-unit in $M_0(O_t,r_t)$ for each $t$. By the definition of $e_n$ (Definition \ref{def:en_esn}), we have the identity
$$
\frac{V(E_k)}{E_k} = \frac{V(e_n)}{e_n} \cdot \left( \frac{V(E_{p-1})}{E_{p-1}} \right)^n ,\leqno{(\ast)}
$$
and we then see that the Coleman--Wan theorem on $V(E_{p-1})/E_{p-1}$ implies that $V(E_k)/E_k$ is a $1$-unit in $M_0(O_t,r_t)$ for each $t$.

As we had $(O_t,r_t)=(\Z_p,1)$ for one $t$, we now also know that $V(E_k)/E_k \in M_0(\Z_p,1)$.

If we combine the information above with Proposition \ref{prop:overconv_greater_than_or_equal_rho} we can now conclude that
$$
\frac{V(E_k)}{E_k} \in M_0\left(\Z_p, \ge \frac{2}{3} \cdot \frac{1}{p+1} \right).
$$

On the other hand, by Theorem \ref{thm:es_n_oc} we have $pe_n \in M_0\left(\Z_p,\ge \frac{p}{p+1}\right)$. By Proposition \ref{prop:overconv_greater_than_or_equal_rho} we can choose a sequence of finite extensions $K_t/\Q_p$ with rings of integers $O_t$ and elements $r_t\in O_t$ such that $pe_n\in M_0(O_t,r_t^p)$, and where $v_p(r_t^p)$ converges to $\frac{p}{p+1}$ from below so that $v_p(r_t)$ converges to $\frac{1}{p+1}$ from below.

As above, the property of the Frobenius operator then implies that $pV(e_n) \in M_0(O_t,r_t)$ for each $t$. But as $v_p(r_t)<\frac{1}{p+1} < \frac{2}{3} \cdot \frac{p}{p+1}$, Corollary \ref{cor:en_esn_1-units} gives us that $e_n$ is a $1$-unit in $M_0(O_t,r_t)$.

We conclude that $p\cdot \frac{V(e_n)}{e_n} \in M_0(O_t,r_t)$ for each $t$, and then the identity $(\ast)$ above, combined again with the Coleman--Wan theorem, implies
$$
p\cdot \frac{V(E_k)}{E_k} \in M_0(O_t,r_t)
$$
for each $t$. As we know from before that $V(E_k)/E_k \in M_0(\Z_p,1)$, the first statement of Theorem \ref{thm:V(E)/E_oc} now follows from Proposition \ref{prop:overconv_greater_than_or_equal_rho}.
\end{proof}

\begin{remark}\label{rem:V(E_k)/E_k_1_unit} As the above proof shows, the function $V(E_k)/E_k$ ($k\ge 4$, divisible by $p-1$) is a $1$-unit in $M_0(O,r)$ if $v_p(r) < \frac{2}{3} \cdot \frac{1}{p+1}$ and $O$ is sufficiently large.
\end{remark}

We are now ready to prove Theorem \ref{thm:es}, but prefer to prove Theorems \ref{thm:es} and \ref{thm:e_n_oc_conditional} together as the arguments are similar. However, some additional preparations are necessary for the proof of Theorem \ref{thm:e_n_oc_conditional}. These will be done in the following two subsections.

\subsection{Some abstract algebra}\label{subsec:algebra} Let $R$ be an $\F_{p}$-algebra. We define a shift operator $D$ on the $R$-module of sequences in $R$: if $(s_n)_{n \geq 0}$ is a sequence of elements in $R$, then
$$
D(s) := (s_{n+1})_{n\geq 0}.
$$

Thus, the sequences of elements in $R$ that satisfy linear recurrences are precisely the sequences $s = (s_n)_{n\geq 0}$ for which there exists a monic polynomial $P \in R[X]$ such that
$$
P(D)(s) = 0.
$$

\begin{lem}[Deep recurrences]\label{deep_rec} Let $s = (s_n)_{n\geq 0}$ be a sequence of elements in $R$ satisfying a linear recurrence
$$
s_n = \sum_{i=1}^r A_{i} s_{n-i}
$$
for some $r\in\N$ and all $n \ge r$.

Then for any $t \in \Z_{\geq 0}$ and all $n \in \Z_{\geq rp^t}$, the sequence $s$ also satisfies the linear recurrence
$$
s_n = \sum_{i=1}^r A_{i}^{p^t} s_{n-ip^t}.
$$
\end{lem}

\begin{proof} Let
$$
P(D) = D^r - A_{1} D^{r-1} - \ldots - A_{r-1} D - A_{r}.
$$

Then $P(D)(s) = 0$, and therefore
$$
P(D)^{p^t} (s) = 0.
$$

Since $R$ is an $\F_{p}$-algebra, we have
$$
P(D)^{p^t} = D^{rp^t} - A_{1}^{p^t} D^{(r-1)p^t} - \ldots - A_{r-1}^{p^t} D^{p^t} - A_{r}^{p^t} ,
$$
and the desired recurrence follows.
\end{proof}

We remind the reader of the definition of the ``$p$-adic weight'' $\delta_p(n)$ of an integer $n \in \Z_{\geq 0}$: if $n$ has $p$-adic expansion $n = \sum_{i\geq 0} a_i p^i$ with $a_i \in \{0, 1, \ldots, p-1\}$, we put $\delta_p(n) := \sum_{i\geq 0} a_i$.

\begin{prop}\label{rec_mod_p} Let $A, B$ be formal variables. Consider the sequence $(s_n)_{n\geq 0}$ in $\F_{p}[A,B]$ given by
$$
s_0 = 1, \quad s_n = 0 \quad \mbox{ for } 1 \leq n \leq p,
$$
$$
s_{n} = A s_{n-p} + Bs_{n-p-1} \quad \mbox{ for } n \geq p+1.
$$

For any $n \in \N$, if $\delta_p(n(p-1)) = p-1$ then $s_n = 0$.
\end{prop}

\begin{proof} We proceed by induction on $n$. Clearly, $\delta_p(0) \not = p-1$. If $1 \leq n \leq p$, then the $p$-adic expansion of $n(p-1)$ is
$$
n(p-1) = (p-n) + (n-1)p
$$
and hence $\delta_p(n(p-1)) = p-1$. By definition, $s_n = 0$.

Assume $n \geq p + 1$ is such that $\delta_p(n(p-1)) = p-1$  and that the statement is true up to $n-1$. As $\delta_p(n(p-1)) = p-1$, we can write
$$
n(p-1) = \sum_{i=1}^{p-1} p^{a_i}
$$
where $0\leq a_1 \leq \ldots \leq a_{p-1}$ are nonnegative integers, not necessarily distinct. If we had $a_{p-1} \le 1$ we would deduce $n(p-1) \le p(p-1)$. But since $n\ge p+1$ we have in fact $n(p-1) > p(p-1)$. We must therefore have $a_{p-1} \ge 2$. Let
$$
n' := n - p^{a_{p-1}-1},
$$
$$
n'' := n - (p+1)p^{a_{p-1}-2}
$$

Clearly, $0\le n' < n$ and $0\le n'' < n$. Moreover,
$$
n'(p-1) = n(p-1) - p^{a_{p-1}} + p^{a_{p-1}-1} = \left(\sum_{i=1}^{p-2}p^{a_i}\right) + p^{a_{p-1}-1},
$$
$$
n''(p-1) = n(p-1) - p^{a_{p-1}} + p^{a_{p-1}-2} = \left(\sum_{i=1}^{p-2}p^{a_i}\right) + p^{a_{p-1}-2},
$$
and so $\delta_p(n'(p-1)) = \delta_p(n''(p-1)) = p-1$. Thus, $s(n') = s(n'') = 0$ by the induction hypothesis. Using Lemma \ref{deep_rec} we then obtain
\begin{eqnarray*}
s_n &=& A^{p^{a_{p-1}-2}}s_{n-p\cdot p^{a_{p-1}-2}} + B^{p^{a_{p-1}-2}}s_{n-(p+1)p^{a_{p-1}-2}} \\
&=& A^{p^{a_{p-1}-2}}s_{n'} + B^{p^{a_{p-1}-2}}s_{n''} = 0 .
\end{eqnarray*}
\end{proof}

Consider the algebra $\Q_{p}[y_1, \ldots, y_{p+1}]$ where $y_1, \ldots, y_{p+1}$ are formal variables. Set $x_0 := 1$ and
$$
x_n := \frac{1}{n}\sum_{i=1}^n (-1)^{i-1}x_{n-i}y_i
$$
for $n=1, \ldots, p+1$. For integers $n \ge p+2$, define
$$
y_{n} := \sum_{i=1}^{p+1} (-1)^{i-1}x_i y_{n-i}.
$$

Define a $\Q_{p}$-algebra homomorphism
$$
\Phi: \Q_{p}[y_1, \ldots, y_{p+1}] \rightarrow  \Q_{p}[t_1, \ldots, t_{p+1}]
$$
by
$$
y_n \mapsto pt_n \quad \mbox{ for } 1 \leq n \leq p,
$$
and
$$
y_{p+1} \mapsto t_{p+1}.
$$

\begin{lem}\label{phi_integrality} We have $\Phi(x_n) \in \Z_p[t_1,\ldots,t_{p+1}]$ for $1\le n\le p+1$, and consequently
$$
\Phi(y_n) \in \Z_{p}[t_1, \ldots, t_{p+1}].
$$
for all $n\in\N$.
\end{lem}
\begin{proof} It is enough to check that $\Phi(x_n) \in \Z_{p}[t_1, \ldots, t_{p+1}]$ for all $1 \leq n \leq p+1$. This is clear from the definition of $x_n$ if $1 \le n \le p-1$ since $n$ is then prime to $p$. We also check
$$
\Phi(x_p) = \frac{1}{p}\sum_{i=1}^p(-1)^{i-1}\Phi(y_i) \Phi(x_{p-i}) = \sum_{i=1}^p(-1)^{i-1} t_i \Phi(x_{p-i}) \in \Z_{p}[t_1, \ldots, t_{p+1}].
$$

Thus $\Phi(x_{p+1}) \in \Z_{p}[t_1, \ldots, t_{p+1}]$ as well.
\end{proof}

Denote by $\overline{\quad\cdot\quad}$ the mod $p$ reduction map from $\Z_{p}[t_1, \ldots, t_{p+1}]$.

\begin{lem}\label{phi_mod_p} For $1 \leq n \leq p+1$, we have
$$
\overline{\Phi(x_n)} = \begin{cases}0 \indent \mbox{if } 1 \leq n \leq p-1,\\ \overline{t_p} \indent \mbox{if } n = p, \\ -\overline{t_{p+1}} \indent \mbox{if } n = p+1.
 \end{cases}
$$

Consequently, we have for all $n \ge p+2$,
$$
\overline{\Phi(y_n)} = \overline{t_p}\cdot\overline{\Phi(y_{n-p})} + \overline{t_{p+1}}\cdot\overline{\Phi(y_{n-(p+1)})}.
$$
\end{lem}
\begin{proof} A straightforward calculation.
\end{proof}

\begin{prop}\label{symm_poly_prop} Let $n \in \N$. If $\delta_p(n(p-1)) = p-1$, then
$$
\Phi(y_n) \in p\Z_{p}[t_1, \ldots, t_{p+1}].
$$
\end{prop}

\begin{proof} Define a sequence in $\F_{p}[\overline{t_p}, \overline{t_{p+1}}]$ by
$$
s_0 = 1, \quad s_n = 0 \quad \mbox{ for } 1 \leq n \leq p,
$$
$$
s_{n} = \overline{t_p} s_{n-p} + \overline{t_{p+1}}s_{n-(p+1)} \quad \mbox{ for } n \geq p+1.
$$

Using the definition of $\Phi$, of the $y_n$, as well as Lemma \ref{phi_mod_p}, we find that $s_n = \overline{\Phi(y_n)}$ for all $n \in \N$. Applying Proposition \ref{rec_mod_p} we then find $\overline{\Phi(y_n)} = 0$ if $n \in \N$ is such that $\delta_p(n(p-1)) = p-1$.
\end{proof}

\subsection{Integrality of \texorpdfstring{$\T_{p,n}(1)$}{Tnp(1)}}\label{subsec:integrality_of_en}

\begin{prop}\label{Tp_oc} Assume that
$$
\T_{p,n}(1) \in M_0\left(\Z_p, \ge \frac{p}{p+1}\right)
$$
for all $n \in \{1, 2, \ldots, p\}$.

Then
$$
\T_{p,n}(1) \in M_0\left(\Z_p, \ge \frac{p}{p+1}\right)
$$
for all $n \in \N$ such that $\delta_p(n(p-1)) = p-1$.
\end{prop}

\begin{proof} Recall the definition of $\T_{p,n}(1)$:
$$
\T_{p,n}(1) = \frac{U(E_{p-1}^n) + p^{n(p-1)-1}V(E_{p-1}^n)}{E_{p-1}^n} .
$$

Consider first the classical Hecke operator $T_p$ acting on forms $f$ of weight $k$ and level $1$: for such $f$ we have
$$
p T_p(f)(z) = p^k f(pz) + \sum_{j=0}^{p-1} f\left( \frac{z+j}{p} \right) ,
$$
and so
$$
p T_p(E_{p-1}^n) = \sum_{j=0}^{p} f_j^n \in M_{n(p-1)}(\Z_p)
$$
for $n\in\N$, if we define $f_j(z) := E_{p-1}((z+j)/p)$ for $j=0,\ldots,p-1$, and $f_p(z) := p^{p-1} E_{p-1}(pz)$. We have a natural embedding $M_{n(p-1)}(\Z_p) \hookrightarrow M_{n(p-1)}(\Z_p,\ge \frac{p}{p+1})$ (\cite[Theorem 3.2]{katz_padic} combined with the definition of $M_{n(p-1)}(\Z_p,\ge \frac{p}{p+1})$; see also the remark after the proof.) The action of
$$
pT_p = pU + p^{n(p-1)}V
$$
commutes with this embedding, and so we see that
$$
p\T_{p,n}(1) = \sum_{j=0}^{p} g_j^n
$$
with $g_j := f_j/E_{p-1}$, $j=0,\ldots,p$. Lemmas \ref{integrality_lemma} and \ref{suppintlemma} imply that
$$
p\T_{p,n}(1) = \sum_{j=0}^{p} g_j^n =: \upsilon_n \in M_0\left( \Z_p,\ge \frac{p}{p+1}\right) .
$$

Put $R:=M_0\left( \Z_p,\ge \frac{p}{p+1}\right)$, an integral domain. Our hypothesis implies that $\upsilon_n \in pR$ for $n=1,\ldots,p$, and we are done if we can show $\upsilon_n \in pR$ whenever $n\ge p+2$ with $\delta_p(n(p-1)) = p-1$. To ease notation, for $n=1,\ldots,p$ denote by $\frac{1}{p} \upsilon_n$ the uniquely determined elements $\tau_n\in R$ such that $\upsilon_n = p\tau_n$.

Put $\xi_0 := 1$ and
$$
\xi_n := \frac{1}{n} \sum_{i=1}^{n} (-1)^{i-1} \xi_{n-i} \upsilon_i \in \Q_p[\upsilon_1,\ldots,\upsilon_{p+1}]
$$
for $n=1,\ldots,p+1$. Then Newton's identities relating sums of powers and symmetric polynomials imply that the $\xi_n$ are the elementary symmetric polynomials in $g_0,\ldots,g_p$, and that for $n\ge p+2$,
$$
\upsilon_n = \sum_{i=1}^{p+1} (-1)^{i-1} \xi_i \upsilon_{n-i} .
$$

As in the setup before Lemma \ref{phi_integrality}, let $y_1,\ldots,y_{p+1}$ and $t_1,\ldots,t_{p+1}$ be formal variables, and let $\Phi$ be the $\Q_p$-algebra homomorphism
$$
\Q_p[y_1,\ldots,y_{p+1}] \rightarrow \Q_p[t_1,\ldots,t_{p+1}]
$$
given by $y_n \mapsto pt_n$ for $n=1,\ldots,p$, and $y_{p+1} \mapsto t_{p+1}$. Define also $x_0 := 1$ and $x_n := \frac{1}{n}\sum_{i=1}^n (-1)^{i-1}x_{n-i}y_i$ for $n=1, \ldots, p+1$ as well as $y_n := \sum_{i=1}^{p+1} (-1)^{i-1}x_i y_{n-i}$ for $n\ge p+2$.

We have a commutative diagram of $\Q_p$-algebra homomorphisms
$$
\xymatrix{ \Q_p[y_1,\ldots,y_{p+1}] \ar@{->}[r]^{\Phi} \ar@{->}[d]_{\phi} & \Q_p[t_1,\ldots,t_{p+1}] \ar@{->}[d]^{\psi} \\
\Q_p[\upsilon_1,\ldots,\upsilon_{p+1}] \ar@{->}[r]^{Id} & \Q_p[\upsilon_1,\ldots,\upsilon_{p+1}]
}
$$
where $\phi(y_n) := \upsilon_n$, $n=1,\ldots,p+1$, and $\psi(t_n) := \frac{1}{p} \upsilon_n$ for $n=1,\ldots,p$, $\psi(t_{p+1}) := \upsilon_{p+1}$.

We then see that $\phi(x_n) = \xi_n$, $n=1,\ldots p+1$, and then $\phi(y_n) = \upsilon_n$ for all $n$. Now Proposition \ref{symm_poly_prop} implies that if $n\ge p+2$ with $\delta_p(n(p-1)) = p-1$ then
$$
\upsilon_n = \phi(y_n) = (\psi \circ \Phi)(y_n) \in p\Z_p\left[ \frac{1}{p} \upsilon_1,\ldots, \frac{1}{p} \upsilon_p,\upsilon_{p+1}\right] \subseteq pR ,
$$
and we are done.
\end{proof}

\begin{remark} Strictly speaking, in order to use \cite[Theorem 3.2]{katz_padic} in the above argument, we should technically increase the level to, say, $N=3$. We can do this without a problem throughout the entire argument as the conclusion is ultimately a statement about the valuations of the coefficients in the Katz expansion of $\T_{p,n}(1)$. Proving that statement at a (technically) increased level is sufficient.
\end{remark}

\subsection{Proofs of Theorems \ref{thm:es} and \ref{thm:e_n_oc_conditional}}\label{subsec:es} Let us briefly recall the definition and some properties of $p$-adic Eisenstein series. We use \cite[Section 1.6]{serre_zeta} as our basic reference and will also (largely) follow the notation there.

For $d\in\Z_p^{\times}$ let $\langle d\rangle$ denote the $1$-unit part of $d$, and let $\tau$ denote the composition of reduction modulo $p$ and the Teichm\"uller character. We then have $d = \langle d\rangle \tau(d)$ for $d\in\Z_p^{\times}$. We shall identify the group of $p$-adic characters of $\Z_p^{\times}$ with $\Z_p \times \Z/(p-1)\Z$ where $\kappa = (s,i) \in \Z_p \times \Z/(p-1)\Z$ is identified with the character given by
$$
\kappa(d) = \langle d\rangle^s \tau(d)^i .
$$

We shall specialize the discussion to the cases where $i=0$ and $s\in\N$ which are what we need in the following. For such a character $\kappa = (s,0)$ we have the (non-normalized) $p$-adic Eisenstein series $G_{\kappa}^{\ast}$ with $q$-expansion
$$
G_{\kappa}^{\ast}(q) = a_0 + \sum_{n=1}^{\infty} \sigma_{\kappa-1}^{\ast}(n) q^n
$$
where
$$
\sigma_{\kappa-1}^{\ast}(n) = \sum_{\substack{d\mid n \\ p\nmid d}} \kappa(d) d^{-1} = \sum_{\substack{d\mid n \\ p\nmid d}} d^{s-1} \tau(d)^{-s}
$$
and for the constant term $a_0$ we have
$$
a_0 = \frac{1}{2} \lim_{m\rightarrow \infty} \zeta(1-k_m)
$$
where $(k_m)$ is a sequence of even integers $k_m \ge 4$ such that $k_m \equiv 0 \pmod{p-1}$, $k_m\rightarrow \infty$ in $\R$, and $k_m \rightarrow s$ in $\Z_p$. The Eisenstein series $G_{\kappa}^{\ast}$ is then the limit in $q$-expansions of the classical (non-normalized) Eisenstein series $G_{k_m}$ with $q$-expansions
$$
G_{k_m}(q) = -\frac{B_{k_m}}{2k_m} + \sum_{n=1}^{\infty} \sigma_{k_m-1}(n) q^n
$$
where $\sigma_{k_m-1}(n) = \sum_{d\mid n} n^{k_m-1}$ as usual.

The constant term $a_0$ evaluates to
$$
a_0 = -\frac{1}{2s} \cdot B_{s,\tau^{-s}},
$$
cf.\ e.g.\ \cite[Theorem 5.11]{washington}. As we have
$$
B_{s,\tau^{-s}} \equiv -\frac{1}{p} \pmod{\Z_p}  ,
$$
e.g., \cite[Exercise 7.6]{washington}, we obtain the normalized Eisenstein series $E_{(s,0)}^{\ast} := \frac{1}{a_0} G_{\kappa}^{\ast}$ with $p$-integral $q$-expansion
$$
E_{(s,0)}^{\ast}(q) = 1 - \frac{2s}{B_{s,\tau^{-s}}} \sum_{n=1}^{\infty} \left( \sum_{\substack{d\mid n \\ p\nmid d}} d^{s-1} \tau(d)^{-s} \right) q^n
$$
which is the $p$-adic limit of the equally $p$-adically integral $q$-expansions
$$
E_{k_m}(q) = 1 - \frac{2k_m}{B_{k_m}} \sum_{n=1}^{\infty} \sigma_{k_m-1}(n) q^n
$$
of the normalized, classical Eisenstein series $E_{k_m}$.

\begin{proof}[Proofs of Theorems \ref{thm:es} and \ref{thm:e_n_oc_conditional}] Consider first Theorem \ref{thm:es}. Let $s\in\N$. Choosing a sequence $(k_m)$ of even integers $\ge 4$ as above we have $E_{k_m} \rightarrow \Es_{(s,0)}$ $p$-adically in $q$-expansions, and so we see that we also have
$$
\frac{V(E_{k_m})}{E_{k_m}} \rightarrow \frac{V(\Es_{(s,0)})}{\Es_{(s,0)}}
$$
$p$-adically in $q$-expansions. As the $k_m$ are divisible by $p-1$, Theorem \ref{thm:V(E)/E_oc} implies
$$
\frac{V(E_{k_m})}{E_{k_m}} \in \frac{1}{p} M_0\left( \Z_p,\ge \frac{1}{p+1} \right) \bigcap M_0\left( \Z_p, \ge \frac{2}{3} \cdot \frac{1}{p+1} \right)
$$
for all $m$, and the claims of Theorem \ref{thm:es} follow from Proposition \ref{prop:conv_in_basis_closedness} (ii).

Let us then turn to Theorem \ref{thm:e_n_oc_conditional}. The first statement, about $e_n$, follows by combining Proposition \ref{comparing_Tp_en} with Proposition \ref{Tp_oc}: if we have $e_n \in M_0\left(\Z_p , \geq \frac{p}{p+1}\right)$ for $n=1,\ldots,p$, Proposition \ref{comparing_Tp_en} implies $\T_{p,n}(1) \in M_0\left(\Z_p , \geq \frac{p}{p+1}\right)$ for $n=1,\ldots,p$; but then Proposition \ref{Tp_oc} implies $\T_{p,n}(1) \in M_0\left(\Z_p , \geq \frac{p}{p+1}\right)$ for all $n$ such that $\delta_p(n(p-1)) = p-1$, and then Proposition \ref{comparing_Tp_en} implies $e_n \in M_0\left(\Z_p , \geq \frac{p}{p+1}\right)$ for all such $n$.

The second statement of Theorem \ref{thm:e_n_oc_conditional}, about $V(E_k)/E_k$ for integers $k\ge 4$ with $\delta_p(k) = p-1$ now follows from the statement about $e_n$: as $\delta_p(k) = p-1$ we have $k\equiv 0 \pmod{p-1}$. Putting $n:=k/(p-1)$ we then have $\delta_p(n(p-1)) = p-1$ and so $e_n \in M_0\left(\Z_p , \ge \frac{p}{p+1}\right)$ by the first statement. We can now argue exactly as in the second part of the proof of Theorem \ref{thm:V(E)/E_oc} above, replacing $pe_n$ by $e_n$ everywhere in the argument. That argument will then show
$$
\frac{V(E_k)}{E_k} \in M_0\left(\Z_p , \ge \frac{1}{p+1}\right) .
$$

Turning now to the final statement of Theorem \ref{thm:e_n_oc_conditional}, about $V(\Es_{(s,0)})/\Es_{(s,0)}$, let $s\in\N$ with $\delta_p(s) < p-1$. Choose $t\in\N$ such that $s<p^t$ and consider the sequence
$$
k_m = s + (p-1-\delta_p(s)) \cdot p^{m+t} .
$$

We see that $(k_m)$ is a sequence of even integers $\ge 4$ with $\delta_p(k_m) = p-1$ for all $m$ and so we now know that
$$
\frac{V(E_{k_m})}{E_{k_m}} \in M_0\left(\Z_p , \geq \frac{1}{p+1}\right)
$$
for all $m$. But, $k_m\rightarrow \infty$ in $\R$, and $k_m\rightarrow s$ in $\Z_p$, and so we find that
$$
\frac{V(\Es_{(s,0)})}{\Es_{(s,0)}} \in M_0\left(\Z_p , \ge \frac{1}{p+1}\right)
$$
by the same reasoning as in the proof of Theorem \ref{thm:es}, i.e., by referring back to Proposition \ref{prop:conv_in_basis_closedness} (ii).
\end{proof}

\section{Computations}\label{sec:comp} To verify Condition \ref{condition:e_n_for_small_n} for a specific prime we must study the overconvergence of the modular function $e_n := E_{n(p-1)}/E_{p-1}^n$ for $n=2,\ldots,p$. Specifically, we must verify that $e_n \in M_0(\Z_p,\ge p/(p+1))$ for these $n$. To verify this for a given $n$, it suffices to compute the (finite) Katz expansion of $e_n$ (Proposition \ref{prop:katz_exp_classical_form}) and then check that we have $v_p(b_i) \ge \frac{p}{p+1} \cdot i$ for the coefficients $b_i$ of the expansion. To compute the expansion, we must of course first get a basis for the $\Z_p$-modules $B_i$ of our chosen splittings
$$
M_{i(p-1)}(\Z_p) = E_{p-1} \cdot M_{(i-1)(p-1)}(\Z_p) \oplus B_i(\Z_p) .
$$

At the beginning of section \ref{sec:topology_overconv_mod_fcts}, in order to be able to refer explicitly back to Katz' original definition of these expansions at level $1$ (\cite[Proposition 2.8.1]{katz_padic}), we required the splittings to have arisen by taking invariants of splittings at a higher, auxiliary level. However, an elementary consideration shows that if the coefficients $b_i$ of the expansion of $e_n$ satisfy $v_p(b_i) \ge \frac{p}{p+1} \cdot i$ for all $i$, then if we choose another system of splittings, say with modules $B_i'(\Z_p)$, then $e_n$ has an expansion w.r.t.\ the new system of splittings where the  coefficients satisfy the same lower bounds.

Thus, for the computational verification of Condition \ref{condition:e_n_for_small_n} for a specific $p$, we are free to choose any system of splittings and in particular we may choose one that is suitable for efficient computation.

We now describe the choice of modules $B_i(\Z_p)$ that we have used. We have followed the idea of Lauder \cite{lauder_comp} of utilizing the fact that we have (upper triangular) ``Miller bases'' of spaces of classical modular forms on $\SL_2(\Z)$:

Let $k$ be a non-negative, even integer and put:
$$
d_k := \left\lfloor \frac{k}{12}\right\rfloor + \begin{cases} 1   \indent \mbox{ if } k \not \equiv 2 \pmod{12} \\
0 \indent \mbox{ if } k \equiv 2 \pmod{12},\end{cases}
$$
so that $d_k$ is the dimension of the classical space of modular forms of weight $k$ on $\SL_2(\Z)$. Put also:
$$
\epsilon(k) := \begin{cases} 0 \indent \mbox{ if } k \equiv 0 \pmod{4} \\
1 \indent \mbox{ if } k \equiv 2 \pmod{4}.\end{cases}
$$

One checks that for $j=0,\ldots,d_k-1$ the numbers $a = \frac{k - 12j - 6\epsilon(k)}{4}$ are non-negative integers and that for each of these $j$ the modular form $\Delta^j E_4^a E_6^{\epsilon(k)} \in M_k(\Z)$ has $q$-expansion starting with $q^j$. Thus, these forms form a $\Z$-basis for the free $\Z$-module $M_k(\Z)$.

Specialize now to weights divisible by $p-1$. For a fixed $i\ge 0$ and $j=0,\ldots d_{i(p-1)}-1$ let
$$
g_{i,j} := \Delta^j E_4^a E_6^{\epsilon(i(p-1))}
$$
again with
$$
a = \frac{i(p-1) - 12j - 6\epsilon(i(p-1))}{4} .
$$

Put $B_0(\Z_p) := \Z_p = M_0(\Z_p)$ and denote for $i \ge 1$ by $B_i(\Z_p)$ the $\Z_p$-submodule of $M_{i(p-1)}(\Z_p)$ spanned by
$$
\mathcal{B}_i= \{g_{i,j} \mid~ d_{(i-1)(p-1)} \le j \leq d_{i(p-1)}-1  \}.
$$

Again by the properties of the $q$-expansions of the $g_{i,j}$ it is clear that $\mathcal{B}_i$ is an $\Z_p$-basis for $B_i(\Z_p)$.

\begin{lem} For each $i \in \N$, we have a direct sum decomposition
$$
M_{i(p-1)}(\Z_p) = E_{p-1} \cdot M_{(i-1)(p-1)}(\Z_p) \oplus B_i(\Z_p) .
$$
\end{lem}

\begin{proof} Let $f \in M_{i(p-1)}(\Z_p)$. Contemplating again the properties of the $q$-expan\-sions of the $g_{i-1,j}$ (and the fact that the $q$-expansion of $E_{p-1}$ starts with $1$), we see that there are $c_j\in\Z_p$ for $j=0,\ldots,d_{(i-1)(p-1)}-1$ such that the $q$-expansion of $f-h$ where
$$
h = \sum_{j=0}^{d_{(i-1)(p-1)}-1} c_j E_{p-1} g_{i-1,j}
$$
starts with $q^t$ for some $t\ge d_{(i-1)(p-1)}$.

But then $h\in E_{p-1} \cdot M_{(i-1)(p-1)}(\Z_p)$ and $f-h\in B_i(\Z_p)$. Thus,
$$
M_{i(p-1)}(\Z_p) = E_{p-1} \cdot M_{(i-1)(p-1)}(\Z_p) + B_i(\Z_p) ,
$$
and the sum is seen to be direct again via consideration of $q$-expansions.
\end{proof}

Using these modules $B_i(\Z_p)$ we have verified Condition \ref{condition:e_n_for_small_n} for $5\le p\le 97$.

We used SageMath \cite{sage} for the actual computations. The code used is publicly available, cf.\ \cite{sage_code}.

\subsection{Two numerical examples}\label{subsec:numerical_examples} An inspection of the proofs of the results of this paper will show that the decisive theorem is Theorem \ref{thm:es_n_oc} in the sense that the other theorems ultimately derive from that. Of the two statements of Theorem \ref{thm:es_n_oc}, the first, i.e., that $e_n$ (and $\es_n$) is in $\frac{1}{p} M_0(\Z_p,\ge \frac{p}{p+1})$, is the most precise as the second derives from that (in combination with Lemma \ref{elem_cong_eis_p2}.) We now give an example showing that the factor $\frac{1}{p}$ can not in general be removed:

Let $p=5$ and consider the Eisenstein series $E_{24}$. A computation in PARI/GP, \cite{PARI2}, shows that we have
$$
E_{24} = E_4^6 + c_1 E_4^3 \Delta + c_2 \Delta^2
$$
with
$$
c_1 = - \frac{340364160000}{236364091} ,\quad c_2 = \frac{30710845440000}{236364091} .
$$

Thus, the Katz expansion of $e_{24}$ is
$$
e_{24} = 1 + \frac{b_3}{E_4^3} + \frac{b_6}{E_4^6}
$$
with $b_3 := c_1 \Delta \in B_3(\Z_5)$, $b_6 := c_2 \Delta^2 \in B_6(\Z_5)$. We find $v_5(b_3) = v_5(b_6) = 4$, and so a numerical illustration of Theorem \ref{thm:es_n_oc}: we have $5\cdot e_{24} \in M_0(\Z_5, \ge \frac{5}{6})$. However, the first two terms in the above sum are in $M_0(\Z_5, \ge \frac{5}{6})$ while the third is not; it follows that $e_{24} \not\in M_0(\Z_5, \ge \frac{5}{6})$.

Keeping $p=5$ and considering now $V(E_{24})/E_{24}$ we can compute the expansion of this function in terms of powers of the function
$$
t(z) := \left( \frac{\eta(5z)}{\eta(z)} \right)^6
$$
with $\eta$ the Dedekind $\eta$-function. The function $t$ is a hauptmodul for the group $\Gamma_0(5)$ the $q$-expansion of which begins with $q$. It is computationally simple to compute the beginning of the expansion
$$
\frac{V(E_{24})}{E_{24}} = \sum_{i=0}^{\infty} a_i t^i
$$
where the $a_i$ will be in $\Z_5$. We find the $5$-adic valuations of $a_0,\ldots,a_{10}$ to be this vector: $(0,1,1,3,3,4,4,5,5,6,4)$. By \cite[Corollary 2.2]{loeffler} we know that if we had $V(E_{24})/E_{24} \in M_0(\Z_p,\ge 1/(p+1))$ then we would have $v_5(a_i) \ge i/2$ for all $i$. Since $v_5(a_{10}) = 4$, we conclude that $V(E_{24})/E_{24} \not\in M_0(\Z_p,\ge 1/(p+1))$. Hence the Coleman--Wan theorem on rate of overconvergence of $V(E_{p-1})/E_{p-1}$ does not extend to the same statement about $V(E_k)/E_k$ for all weights $k$ divisible by $p-1$.

However, the beginning of the sequence $v_5(a_i)$ illustrates the first statement of Theorem \ref{thm:V(E)/E_oc}, i.e., that we have $p\cdot V(E_k)/E_k \in M_0(\Z_p,\ge 1/(p+1))$ in general for weights $k\ge 4$ divisible by $p-1$.

\section{Sample applications}\label{sec:literature} As described in the introduction, establishing explicit rates of overconvergence for the $p$-adic modular function $V(\Es_{(1,0)})/\Es_{(1,0)}$, or, perhaps other members of the ``Eisenstein family'', is a crucial tool for studying finer details of the Coleman-Mazur eigencurve.

One may ask why it would matter that this function or other functions can be established to have an explicit rate of overconvergence in contrast with just knowing that it has a positive, but unknown rate. A first answer to this can be found with the quest to find good, explicit lower bounds for the Newton polygon of the $U$ operator on forms of weight $k$. As an example of this one can consider section 5 of Wan's paper \cite{wan} where an explicit lower bound valid for all even weights is worked out for the case that $p\equiv 1\pmod{12}$ and the tame level is $1$. An inspection of the argument will show that this explicit lower bound really depends on the input from the Coleman--Wan theorem, i.e., the information about rate of overconvergence of the function $\frac{V(E_{p-1})}{E_{p-1}}$. For primes $p$ that satisfy Condition \ref{condition:e_n_for_small_n} of the introduction, so that we know by Theorem \ref{thm:e_n_oc_conditional} that
$$
V(\Es_{(1,0)})/\Es_{(1,0)} \in M_0\left( \Z_p,\ge \frac{1}{p+1} \right) ,
$$
it is possible to improve that lower bound. Taking as an example again primes $p\equiv 1\pmod{12}$, one finds for instance that the valuation of the $n$th coefficient of the characteristic series of $U$ acting on forms of level $1$ and even weight $k$ is bounded from below by
$$
\frac{6}{p+1} \cdot (n-1)^2 - n .
$$

For the primes $p=5,7,13$ (that satisfy Condition \ref{condition:e_n_for_small_n}) it is possible to improve these lower bounds even further via a more delicate analysis along the lines of the arguments of \cite{cst} (the method here depends on $X_0(p)$ having genus $0$): for each of these three primes one can improve the lower bound to
$$
\frac{6}{p+1} \cdot n^2 - \left( 1- \frac{6}{p+1} \right) \cdot n - 1,
$$
with the arguments again conditional on the above information on $V(\Es_{(1,0)})/\Es_{(1,0)}$.

Among other things, the thesis \cite{destefano} used this lower bound together with some additional argumentation to prove a number of conditional results such as for example the following.

\begin{theorem} (\cite[Theorem 4.30]{destefano}) Let $p=5$. Given the result on $V(E^{\ast}_{(1,0)})/E^{\ast}_{(1,0)}$ of Theorem \ref{thm:e_n_oc_conditional}, the following holds. There is a $p$-adic analytic family of modular forms
$$
f_k(q) = q + \sum_{n=2}^{\infty} a_n(k) q^n
$$
that, when specialized to integers $k_0\ge 4$ with
$$
k_0\equiv 2 \pmod{4}, ~~ v_5(k_0-10) \le 1, ~~\mbox{and} ~~ v_5(k_0-14)\le 1 ,
$$
gives us a classical, cuspidal eigenform of weight $k_0$ on $\Gamma_0(5)$ and slope
$$
2 + v_5((k_0-10)(k_0-14))
$$
which is the smallest possible slope among such eigenforms of weight $k_0$. Furthermore, $f_{k_0}$ is the unique normalized eigenform of weight $k_0$ on $\Gamma_0(5)$ with this slope.

If $k_0,k_1 \ge 2$ are classical weights like that, we have
\begin{eqnarray*}
&& v_5(a_n(k_0) - a_n(k_1)) \ge v_5(k_0-k_1) + 1 \\
&+& \min \{ 0, 1-v_5((k_0-10)(k_1-10)) \} + \min \{ 0, 1-v_5((k_0-14)(k_1-14)) \} .
\end{eqnarray*}
\end{theorem}

The thesis \cite{destefano} contains several other such results, also for $p=7$, and other congruence classes of weights modulo $p-1$. Because of Theorem \ref{thm:e_n_oc_conditional} (for $p=5,7$) these statements are now unconditional theorems.

As we mentioned in the introduction, explicit information on overconvergence rates of the types of modular functions considered in this paper seems to have been available only for the primes $2$ and $3$:

For $p=2$, a central result and tool of the paper \cite{buzzard_kilford_2-adic} by Buzzard and Kilford is the statement that, as formulated in our language, we have
$$
\frac{E^{\ast}_k}{V(E^{\ast}_k)} \in M_0\left(\Z_2, \ge \frac{1}{4}\right) ,
$$
for natural numbers $k$ divisible by $4$. Cf.\ \cite[Corollary 9(iii)]{buzzard_kilford_2-adic}.

Similarly, for $p=3$, a central result and tool of Roe's paper \cite{roe_3-adic} can be reformulated in our language as saying that
$$
\frac{E^{\ast}_k}{V(E^{\ast}_k)} \in M_0\left(\Z_3, \ge \frac{1}{6}\right) ,
$$
for even natural numbers $k$.

In \cite{coleman_eisenstein}, Coleman uses these statements to prove for the primes $p=2,3$ the general conjecture that he formulates in that paper concerning analytic continuation of the Eisenstein family.

Further applications of the results of this paper, possibly including further remarks on Coleman's conjecture on the ``Eisenstein family'' (cf.\ \cite{coleman_eisenstein}) will be reported on elsewhere.

\end{document}